\documentclass[a4paper,12pt]{amsart}

\usepackage{comment}
\usepackage[
  margin=30mm,
  marginparwidth=25mm,     
  marginparsep=2mm,       
  bottom=25mm,
  ]{geometry}
\usepackage{allrunes}

\usepackage[T1]{fontenc}
\usepackage[utf8]{inputenc}
\usepackage[english]{babel}

\usepackage[numbers]{natbib}

\usepackage{tikz}
\usetikzlibrary{shapes.geometric}
\pgfdeclarelayer{bg}
\pgfsetlayers{bg,main}

\usepackage{xcolor}
\usepackage{amssymb,latexsym,amsmath,amsthm,extarrows}
\usepackage{graphicx,mathrsfs}

\usepackage{latexsym}
\usepackage{delarray}
\usepackage{bbm}
\usepackage{hyperref}
\hypersetup{colorlinks,allcolors=black}

\numberwithin{equation}{section}

\newtheorem{Th}{Theorem}[section]
\newtheorem{Lem}[Th]{Lemma}
\newtheorem{Prop}[Th]{Proposition}
\newtheorem{Rem}[Th]{Remark}

\newtheorem{Def}[Th]{Definition}
\newtheorem{Cor}[Th]{Corollary}

\newcommand{\Rl}{\mathbb{R}}
\newcommand{\Z}{\mathbb{Z}}

\newcommand{\Rn}{\mathbb{R}^{n}}

\renewcommand{\d}{\partial}

\newcommand{\dd}{\,\mathrm{d}}

\newcommand{\wt}{\widetilde}

\newcommand{\norm}[1]{\left\Vert#1\right\Vert}
\newcommand{\jap}[1]{\langle #1 \rangle}
\DeclareMathOperator{\supp}{supp}

\newcommand{\eps}{\epsilon}

\allowdisplaybreaks 

\title[Bilinear Sparse Domination]
{Bilinear Sparse Domination for Oscillatory Integral Operators}
\author{Tobias Mattsson}

\address{\newline
 Tobias Mattsson \\
 Department of  Mathematics, Uppsala University \\
 S-751 06 Uppsala, Sweden}
       
\email{tobias.mattsson@math.uu.se}
\date{}

\thanks{A grant from the K. and A. Wallenberg foundation supports the author.}


\begin{document}

\maketitle

\begin{abstract}
   In this paper, we prove bilinear sparse domination bounds for a wide class of Fourier integral operators of general rank, as well as oscillatory integral operators associated to H\"ormander symbol classes $S^m_{\rho,\delta}$ for all $0\leq\rho\leq 1$ and $0\leq\delta< 1$, a notable example is the Schr\"odinger operator. As a consequence, one obtains weak $(1,1)$ estimates, vector-valued estimates, and a wide range of weighted norm inequalities for these classes of operators.
\end{abstract}

\section{Introduction}

In this paper we study the bilinear sparse domination of oscillatory integral operators of the form:
\begin{equation*}
    T_a^\varphi f(x) := \int_{\Rn} e^{i\varphi(x,\xi)}\,a(x,\xi)\,\widehat f (\xi) \dd\xi
\end{equation*}
where $a$ is an amplitude of H\"ormander type, and the phase function class will satisfy regularity conditions in $\xi$ such that the operator is either a Fourier integral operator (FIO), or an oscillatory integral operator (OIO). See below for the definition of these two classes (see section \ref{sec:fio} and \ref{sec:oio}).\\

Recently, a significant advancement in pointwise estimates has been made with the development of the theory of sparse domination. This theory introduces the concept of sparse domination for operators $T$ on function spaces, characterized by the following inequalities:
\begin{align*}
|Tf(x)| \lesssim {\Lambda}_{\mathcal{S}, q}f(x) \quad \text{and} \quad |\langle{Tf,g}\rangle| \lesssim {\Lambda}_{{\mathcal{S}}, q, p'}(f, g).
\end{align*}

We call the first inequality \textit{sparse domination} and the second one \textit{bilinear sparse domination}. See section \ref{sparse domination theory} for the definition of ${\Lambda}_{ \mathcal{S} , q}$ and ${\Lambda}_{{\mathcal{S}},q,p'}$.\\

The significance of proving sparse bounds extends well beyond their relationship to $L^p$ bounds. These bounds are useful for establishing weighted and vector-valued inequalities. This approach has been applied in many contexts, such as Bochner-Riesz multipliers \cite{LMRNC}, singular integrals \cite{LSS,grafakos}, various Hilbert transforms, multipliers and pseudodifferential operators \cite{BC,yamamoto}. Sparse bounds were also employed in the proof of the $A_2$ conjecture \cite{Lacey,Lerner_simple_2012}.\\

In \cite{BC}, D. Beltran and L. Cladek showed that a sharp (up to the endpoint) bilinear sparse domination estimate for classical pseudodifferential operators, with amplitudes in $S^m_{\rho,\delta}(\Rn)$ for $0<\delta\leq \rho< 1$, holds. Building upon this result, our paper delves into the investigation of sparse domination for Fourier integral operators and other oscillatory integral operators with nonlinear phase functions associated to partial differential equations, like the Schr\"odinger equation. In the former case we focus on phase functions belonging to the Dos Santos Ferreira-Staubach classes $\Phi^2$. While in the latter case, we investigate the sparse domination of oscillatory integral operators with phase functions in the class $\textart{F}^k$, which was initially introduced by A. J. Castro, A. Israelsson, W. Staubach, and M. Yerlanov \cite{CISY}.\\

The sparse bounds established in our study yield a range of consequential results, including weighted and vector-valued inequalities for FIOs and OIOs. While some of the weighted norm inequalities have previously been established in the case of Fourier (see D. Dos Santos Ferreira, and W. Staubach \cite{DS}) and oscillatory integral operators (see A. Bergfeldt, and W. Staubach \cite{BS}), our work unveils new weighted norm inequalities, and gives quantitative control of the weighted operator norm. These new results, along with the previously known ones, are discussed in detail in Section \ref{Sec:weightedConsequences}. We choose to only discuss our main results about sparse domination in the introduction, beginning with the Fourier integral operators. \\

    To this end, we define the following notation, set 
    \begin{equation}\label{decay_FIO}
        m_\rho(q,p):=-n(1-\rho)\Big|\frac{1}{q}-\frac{1}{p}\Big|,
    \end{equation}
    and
    $$\zeta:=\min\Big\{0,\frac{n}{2}(\rho-\delta)\Big\}.$$
    This latter number is the sharp regularity exponent in the $L^2$ boundedness of FIOs and also OIOs.
    
    \begin{Th}[Fourier Sparse Domination]\label{main_FIO}
        Assume that $a(x,\xi)\in S^{m}_{\rho,\delta}(\Rn)$ for $0\leq \rho\leq 1$ and $0\leq \delta< 1$, and $\varphi(x,\xi)$ is in the class $\Phi^2$ with rank $0\leq\kappa\leq n-1$ and is \emph{SND}. Then for any compactly supported bounded functions $f, g$ on $\mathbb{R}^n$, there exist sparse collections $\mathcal{S}$ and $\widetilde{\mathcal{S}}$ of dyadic cubes such that
        \begin{align*}
            \begin{cases}
            |\left<T_a^\varphi f, g\right>|\le C(m, q, p)\Lambda_{\mathcal{S}, q, p'}(f, g), \text{ and }\\
            |\left<T_a^\varphi f, g\right>|\le C(m, q, p)\Lambda_{\widetilde{\mathcal{S}}, p', q}(f, g),
            \end{cases}
        \end{align*}
        for all pairs $(q, p')$ and $(p', q)$ such that 
        \begin{align}\label{FIO_ranges}
            \begin{cases}
            m<-m_{\rho}(q,2)-\kappa\rho\Big(\frac{1}{p}-\frac{1}{2}\Big)-\zeta, & 1\le q\le p \le 2\\
            m<-m_{\rho}(q,p)-\zeta\big(\frac{2}{q}-1\big), & 1\le q\le 2 \le p \le q'.
            \end{cases}
        \end{align}
        (See figure \ref{figure numerology})
    \end{Th}

     If the phase $\varphi$ is linear (or $\kappa=0$), the FIOs reduce to the pseudodifferential case, this case was investigated by Beltran and Cladek in \cite{BC}. Their sparse domination result suggests that the bilinear sparse domination estimates for the pseudodifferential operators could be improved in two ways. First, pseudodifferential operators are examples of Fourier integral operators with phase function of rank $0$, secondly the order and type of the amplitude i.e. $m,$ $\rho$ and $\delta$. The motivation for investigating Fourier integral operators with amplitude types different than $\rho=1$ and $\delta=0$, and ranks different than $n-1$, comes from the theory of partial differential equations, scattering theory, inverse problems, and tomography, just to name a few. In this paper, we have made an attempt to investigate and achieve optimal results in all three of these directions.\\

    We now turn to the oscillatory integral operators. Let $\varkappa=\min(\rho,1-k)$ and set 
    \begin{equation}\label{decay_OIO}
       m_\varkappa(q,p):=-n(1-\varkappa)\Big|\frac{1}{q}-\frac{1}{p}\Big|.
    \end{equation}
    
     We also extend the class of the phase function to include nonlinear phase functions of the type $x\cdot\xi+|\xi|^k$. This second class includes various oscillatory integral operators described in more detail below. 
    
    \begin{Th}[Oscillatory Sparse Domination]\label{main_OIO}
        Let $n\geq 1$, $0<k<\infty$. Assume that $\varphi\in \textart F^k$ {is \emph{SND}}, satisfies the \emph{LF}$(\mu)$-condition for some $0<\mu\leq 1$, and the $L^2$-condition \eqref{eq:L2 condition_old}. Assume also that $a(x,\xi)\in S^{m}_{\rho,\delta}(\Rn),$ for $0\leq \rho\leq 1$ and $0\leq \delta< 1$. Then for any compactly supported bounded functions $f, g$ on $\mathbb{R}^n$, there exist sparse collections $\mathcal{S}$ and $\widetilde{\mathcal{S}}$ of dyadic cubes such that
        \begin{align*}
            \begin{cases}
            |\left<T_a^\varphi f, g\right>|\le C(m, q, p)\Lambda_{\mathcal{S}, q, p'}(f, g), \text{ and }\\
            |\left<T_a^\varphi f, g\right>|\le C(m, q, p)\Lambda_{\widetilde{\mathcal{S}}, p', q}(f, g),
            \end{cases}
        \end{align*}
        for all pairs $(q, p')$ and $(p', q)$ such that 
        \begin{align}\label{OIO_ranges}
            \begin{cases}
            m<-m_\varkappa(q,2)-\zeta, & 1\le q\le p \le 2\\
            m<-m_\varkappa(q,p)-\zeta\big(\frac{2}{q}-1\big), & 1\le q\le 2 \le p \le q'.
            \end{cases}
        \end{align}
        (See figure \ref{figure numerology})
    \end{Th}
   
    The motivation behind investigating the OIOs in this paper stems from the theory of partial differential equations, specifically in the study of dispersive equations. These dispersive equations often involve phase functions $\varphi(x, \xi) = x \cdot \xi + \phi(\xi)$, where $\xi$ represents the wave variable and $x$ denotes the spatial variable. Different choices of $\phi(\xi)$ lead to various important equations. For instance, when $\phi(\xi) = |\xi|^{1/2}$, it corresponds to the water-wave equation. When $\phi(\xi) = |\xi|^{2}$, it relates to the Schr\"odinger equation. Furthermore, in dimension one, $\phi(\xi) = |\xi|^{3}$ and $\phi(\xi) = \xi|\xi|$ correspond to the Airy and Benjamin-Ono equations, respectively.\\

    This paper, and the method of obtaining sparse form bounds, is inspired by the work by Beltran and Cladek \cite{BC}, M. T. Lacey and S. Spencer \cite{LSS}, and M. T. Lacey, D. Mena, and M. C. Reguera \cite{LMRNC}. The essential ingredients of this method are geometrically decaying \textit{$L^p$-improving estimates} (i.e. $L^q\to L^p$ estimates for $p>q$) on spatially and frequency localised pieces of the operator, as well as optimal $L^p$ estimates on frequency localized pieces.\\

    For $m<-n(1-\rho)(1/q-1/2)$ in the range $1\leq q\leq p\leq 2$, and $m<-n(1-\rho)(1/q-1/p)$ whenever $1\leq q\leq 2\leq p\leq q'$, Beltran and Cladek obtained sharp up to the endpoint bilinear sparse domination estimates for the pseudodifferential operators with symbols in the H\"ormander classes $S^m_{\rho,\delta}$ for $0<\delta\leq\rho<1$. In the corresponding ranges of $p,q$ we obtain \eqref{FIO_ranges} for the Fourier integral operator with amplitudes in the H\"ormander classes $S^m_{\rho,\delta}$ for $0\leq\delta<1$ and $0\leq\rho\leq 1$. Thus, since Theorem \ref{main_FIO} reduces to the sharp sparse domination result obtained for the pseudodifferential case in \cite{BC} whenever $\rho\geq \delta$ and $\kappa=0$, our result is also sharp in that range. \\

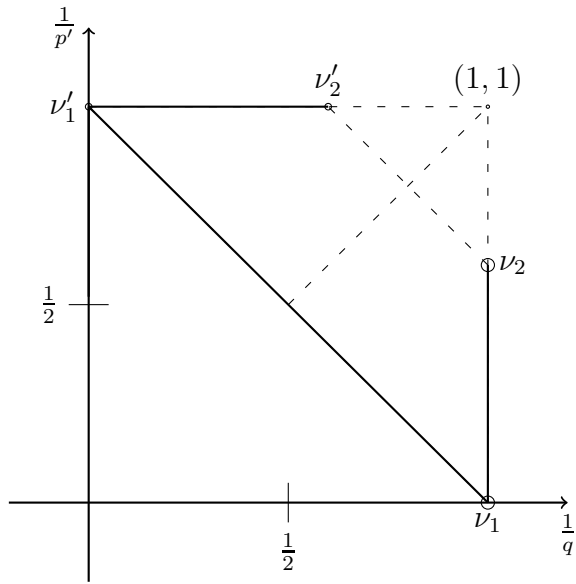
\begin{figure}
\begin{tikzpicture}[scale=3.5] 
\begin{scope}[scale=1.5]
\draw[thick,->] (-.2,0) -- (1.2,0) node[below] {$\frac{1}{q}$};
\draw[thick,->] (0,-.2) -- (0,1.2) node[left] {$\frac{1}{p'}$};
\draw (.5,.05) -- (.5,-.05) node[below] {$\frac{1}{2}$};
\draw (.05,.5) -- (-.05,.5) node[left] {$\frac{1}{2}$};
\draw (1,0) circle (.04em) node[below] {$\nu_1$}; 
\draw (0,1) circle (.02em) node[left] {$\nu_1'$};
\draw[loosely dashed] (0,1) -- (0.5,0.5)  -- (1.,0); 
\draw[loosely dashed]  (0.5,0.5)  -- (1.,1); 
\draw[thick] (0,1) -- (0.6,1);
\draw[loosely dashed] (1,0.6) -- (.6,1); 
\draw[loosely dashed] (0,1) -- (1.,1.)  -- (1.,0); 
\draw[thick] (0,1) -- (0.5,0.5)  -- (1.,0); 
\draw[thick] (0,1) -- (0,0.52);
\draw[thick] (1,0) -- (1,.6);
\draw (1.,.6) circle (.04em) node[right] {$ \nu_2$};
\draw (.6,1) circle (.02em) node[above] {$ \nu_2'$};
\draw (1,1) circle (.01em) node[above] {$(1,1)$};
\end{scope}
\end{tikzpicture}

\caption{The trapezoid $\textbf{T}$ of Theorem \ref{main_OIO}. The thick lines are contained in the diagram while the dashed lines are not.}
\label{sparse region}
\label{figure numerology}
\end{figure}

We shall briefly describe the trapezoid associated with the admissible $(1/q,1/p')$-points corresponding to the case of OIOs for $m<0$. A similar description works to describe the numerology in Theorem \ref{main_FIO}. Let $\textbf{T}$ be the trapezoid with vertices,
\begin{align*}
\textbf{T}:
\begin{cases}
    \nu_1=\big(1,0\big),& \nu_2=\big(1,\frac{-m-\zeta}{n(1-\varkappa)}\big),\\
   \nu_1'=\big(0,1\big), & \nu_2'=\big(\frac{-m-\zeta}{n(1-\varkappa)},1\big).
    \end{cases}
\end{align*}
where $-n(1-\varkappa)\leq m+\zeta\leq-\frac{n(1-\varkappa)}{2}$. While if $-\frac{n(1-\varkappa)}{2}-\zeta<m<0$ then we obtain the trapezoid
\begin{align*}
\textbf{T}':
\begin{cases}
    \mu_1=\big(1/2-\frac{m+\zeta}{n(1-\varkappa)},1/2+\frac{m+\zeta}{n(1-\varkappa)}\big),& \mu_2=\big(1/2-\frac{m+\zeta}{n(1-\varkappa)},1/2\big),\\
   \mu_1'=\big(1/2+\frac{m+\zeta}{n(1-\varkappa)},1/2-\frac{m+\zeta}{n(1-\varkappa)}\big), & \mu_2'=\big(1/2,1/2-\frac{m+\zeta}{n(1-\varkappa)}\big).
    \end{cases}
\end{align*}

See figure \ref{figure numerology} for a visualization of the trapezoid $\textbf{T}$. One may obtain similar trapezoids for the Fourier integral operators, e.g. just substitute $\rho$ for $\varkappa$ in $\textbf{T}$ to to obtain the corresponding trapezoid for the FIOs.\\

\subsection*{Organization and notation}

In  the first section of the paper (section \ref{Prelim}), we introduce the necessary fundamental concepts from microlocal analysis, weighted $A_p$-theory, and sparse domination theory. We also give the definitions of the various classes of amplitudes and phase functions regarding oscillatory integral operators and Fourier integral operators. \\

In section \ref{Sec:weightedConsequences} we derive all the corollaries of Theorem \ref{main_FIO} and Theorem \ref{main_OIO}, these include weighted norm inequalities, weak $(1,1)$ estimates, and vector-valued inequalities. We also discuss how these estimates are new in relation to previously established results.\\

In section \ref{sec:decomp} we provide the various decompositions that we employ in each case, along with various tools that we need in order to obtain the geometrically decaying $L^q$-improving estimates in section \ref{sec:proof_oio}.  This section is concerned with the proof of $L^q$-improving estimates of oscillatory integral operators. This process is divided in two main steps.  The first step is to use the various tools developed in the first section to obtain $L^q$-improving estimates which decay geometrically. In the second step we use the optimal global $L^2$-boundedness of OIOs shown in \cite{IMS2}, along with the geometric decay of the first step to obtain $L^2$ and estimate for the principal term of the decomposition. We do a similar procedure to obtain $L^1$ and $L^1\to L^\infty$ cases for the localized pieces and the principal term. By an interpolation procedure we then obtain the desired $L^q$-improving estimates.\\

In section \ref{sec:proof_fio} we turn to the proof of the $L^q$-improving estimates pertaining to FIOs. We employ a similar strategy to the one in section \ref{sec:proof_oio}, the resulting analysis differs in the details more so than the overall strategy. \\

Finally in section \ref{sec:proof} we give the proof of the main results.\\

As for notation, in the subsequent analysis we will adopt the customary practice of denoting positive constants in inequalities by the symbol $C$. The specific value of $C$ is not crucial to the current problem but can be determined based on known parameters, such as $n$, $p$, and $q$ which are relevant in the given context. For instance, these parameters may be associated with the seminorms of various amplitudes or phase functions. While the value of $C$ may vary from line to line, it can be estimated if required. In this paper, we refer to $C$ as the "hidden constant".\\

Additionally, we adopt the notation $c_1 \lesssim c_2$ as a shorthand for $c_1 \leq Cc_2$. Moreover, we use the notation $c_1 \sim c_2$ when both $c_1 \lesssim c_2$ and $c_2 \lesssim c_1$. 

\subsection*{Acknowledgements.}
The author is supported by the Knut and Alice Wallenberg Foundation. The author is also grateful to Andreas Str\"ombergsson for his support and encouragement.

\section{Preliminaries}\label{Prelim}

We start by recalling the definition of the Littlewood-Paley partition of unity, which is the most basic tool in the frequency decomposition of the operators at hand.\\

Let $\psi_0 \in \mathcal{C}_c^\infty(\mathbb{R}^n)$ be equal to $1$ on $B(0,1)$ and have its support in $B(0,2)$. Then let
\begin{equation*}
\psi_j(\xi) := \psi_0 \left (2^{-j}\xi \right )-\psi_0 \left (2^{-(j-1)}\xi \right ),
\end{equation*}
where $j\geq 1$ is an integer and $\psi(\xi) := \psi_1(\xi)$. Then $\psi_j(\xi) = \psi\left (2^{-(j-1)}\xi \right )$, and one has the following Littlewood-Paley partition of unity
\begin{equation*}
    \sum_{j=0}^\infty \psi_j(\xi) = 1,
\end{equation*}
for all $\xi\in\Rl^n$. We also define the fattened Littlewood-Paley operators, set
\begin{equation*}
\Psi_j := \psi_{j+1}+\psi_j+\psi_{j-1},
\end{equation*}
with $\psi_{-1}:=\psi_{0}$. Define the Littlewood-Paley operators by
\begin{equation*}
\psi_j(D), f(x) = \int_{\mathbb{R}^n} \psi_j(\xi)\widehat{f}(\xi)e^{ix\cdot\xi} \mathrm{d}\xi.
\end{equation*}

In connection to the $L^q\to L^p$ estimates for FIOs and OIOs below, we will encounter the well-known $L^r$-maximal function
 \begin{equation*}\label{HLmax}
  \mathcal{M}_r f(x):=\sup _{B\ni x}\Big(\frac{1}{|B|} \int_{B}|f(y)|^r \dd y\Big)^{1/r},   
 \end{equation*} 
 where the supremum is taken over all balls $B$ containing $x$. It is a classical result of Hardy and Littlewood that the $L^r$-maximal function $\mathcal{M}_rf(x)$ is bounded on $L^p$ for $p>r$.\\

    The class of amplitudes considered in this paper was first introduced by L. H\"ormander in \cite{Hor1}.
    \begin{Def}\label{symbol class Sm}
        Let $m\in \Rl$ and $0\leq \rho, \delta \leq 1$. An amplitude $a(x,\xi)$ in the class $S^m_{\rho,\delta}(\Rn)$ is a function $a\in \mathcal{C}^\infty (\Rn\times \Rn)$ for which
        \begin{equation*}
        \left|\partial_{\xi}^\alpha \partial_{x}^\beta a(x,\xi) \right| \lesssim \langle\xi\rangle ^{m-\rho|\alpha|+\delta|\beta|},
        \end{equation*}
        for all multi-indices $\alpha$ and $\beta$ and $(x,\xi)\in \Rn\times \Rn$, where 
        $\langle\xi\rangle:= (1+|\xi|^2)^{1/2}.$
    \end{Def}
    
    In our subsequent discussion, we will refer to $m$ as the order of the amplitude, while using $\rho$ and $\delta$ to denote its type. Specifically, we will refer to the class $S_{\rho,\delta}^m(\mathbb{R}^n)$, where $0<\rho\leq 1$ and $0\leq \delta<1$, as the classical class. Similarly, the class $S_{0,\delta}^m(\mathbb{R}^n)$ with $0\leq \delta<1$ will be referred to as the exotic class, and we will denote the class $S_{\rho,1}^m(\mathbb{R}^n)$, where $0\leq \rho\leq 1$, as the forbidden class of amplitudes. 

\subsection{Tools from weighted function spaces}

 In deriving the weighted consequences of Theorem \ref{main_FIO} and Theorem \ref{main_OIO}, we will make use of the classical $A_p$ space and the reverse H\"older space $RH_p$. We recall the definitions of these below for the convenience of the reader.

 \begin{Def}
    Let $1<p< \infty$. A locally integrable and a.e. positive function $w$ is called a weight. A weight is called a Muckenhoupt $A_p$ weight if 
    \begin{equation*}
        [w]_{A_p}=\sup_{Q\text{ cube}} \jap{w}_{Q,1}\jap{w^{1-p'}}_{Q,1}^{p-1}<\infty,
    \end{equation*}
    while a weight is called a reverse H\"older $RH_p$ weight if
    \begin{equation*}
        [w]_{RH_p}=\sup_{Q\text{ cube}} \jap{w}_{Q,1}^{-1}\jap{w}_{Q,p}<\infty.
    \end{equation*}
\end{Def}

The intersection of the $A_s$ and the reverse H\"older classes can be characterized by the equivalence:
\begin{equation}\label{equivalenceApRH}
w \in A_{s/q} \cap RH_{(p/s)'} \Longleftrightarrow w^{(p/s)'} \in A_{(p/s)'(s/q-1 )+1}.
\end{equation}

\begin{Def}\label{def:TLspace}
	Let $s \in {\Rl}$ and $1\leq p \leq\infty$, $1\leq q \leq\infty$ and $w\in A_p$. The weighted Triebel-Lizorkin space is defined by
	\[
	F^s_{p,q}(w)
	:=
	\Big\{
	f \in {\mathscr{S}'}(\Rl^n) \,:\,
	\|f\|_{F^s_{p,q}(w)}
	:=
	\|(2^{jqs}\psi_j(D) f)\|_{L^p_w(\ell^q)}<\infty
	\Big\},
	\]
where $\mathscr{S}'(\Rl^n)$ denotes the space of tempered distributions.	
\end{Def}
 
Later, in the discussion on vector-valued inequalities we shall need the following classical lemma due to J. L. Rubio de Francia \cite{RUbio}.

\begin{Def}\label{linearizable}
    An operator $T$ defined in $L^p$ is called linearizable if for any $f_0\in L^p$, there exists a linear operator $\mathcal{L}$  such that
    \begin{align*}
        &|Tf_0|=|\mathcal{L}f_0|; &|\mathcal{L}f|\leq|Tf|,\quad\forall f\in L^p.
    \end{align*}
\end{Def}

\begin{Th}\label{Rubio}
    Let $T_j$ be a sequence of linearizable operators and suppose that for some fixed $r\geq 1$, and all weights $w\in A_r$ one has $\|T_j f\|_{L^r_w}\lesssim \|f\|_{L^r_w}$. Then for $1 < p < \infty$ and $w\in A_p$ one has $$\|T_j f_j\|_{L^p_w(\ell^r)}\lesssim_w \|f_j\|_{L^p_w(\ell^r)}$$
\end{Th}

\subsection{Background on Sparse Forms}\label{sparse domination theory}

Now, let us delve into some facts concerning sparse bounds.

    \begin{Def}
        Let $0<\eta <1$. A collection $\mathcal{S}$ of cubes in $\Rn$ is called an $\eta$-sparse family if there are pairwise disjoint  subsets  ${\{ E_Q\}}_{Q \in \mathcal{S} }$ such that
        \begin{enumerate}
            \item $ E_{Q} \subset Q $
            \item and $|E_{Q}|>\eta |Q|$.
        \end{enumerate}
    \end{Def} 

    If there is no confusion, we usually say \textit{sparse} instead of \textit{$\eta$-sparse}. By a cube $Q$ in $\Rn$ we mean a half open cube such as $$Q=[x_1,x_1+l(Q))\times\cdots[x_n,x_n+l(Q)),$$ where $l(Q)$ is the side length of the cube and $x(Q)=(x_1,\dots,x_n)$ is a corner of the cube. Define a dyadic cube to be a cube $Q$ with $l(Q)=2^k$ for some $k\in\Z$. Let $v$ be vector in $\{0,1,2\}^n$ and define the shifted dyadic grid $\mathcal{D}_v$ as the set 
    $$\{Q:l(Q)=2^k,x(Q)=\overline{m},\forall k\in\Z,\,\overline{m}\in\Z^n\}.$$
    
    \begin{Def}
        For any cube $Q$ and $1\leq p <\infty $, we define $ {\langle f \rangle }_{p,Q} := {|Q|}^{-\frac{1}{p}} {||f||}_{ L^p(Q)} $. Let $\mathcal{S}$ be an $\eta$-sparse family and $1\leq  q <\infty$, then the $(q,s)$-sparse form operator ${\Lambda}_{\mathcal{S} , q,p}$ and $q$-sparse operator $ {\Lambda}_{ \mathcal{S} ,q}$ are defined by
        \begin{eqnarray*}
         {\Lambda}_{ \mathcal{S} ,q}f(x) :=\sum_{Q \in \mathcal{S}} {\langle f \rangle }_{q,Q} 1_Q (x) \ \ , \ \ 
         {\Lambda}_{\mathcal{S} , q,p}(f,g) := \sum_{Q \in \mathcal{S}} |Q| {\langle f \rangle }_{q,Q}{\langle g \rangle }_{p,Q}  
        \end{eqnarray*}
        for all $ f,g \in L^{1}_{loc} $.
    \end{Def}

If $q<s<p$, we have 
    \begin{align*}
        {\Lambda}_{\mathcal{S} , q,p'}(f,g) \lesssim {||f||}_{s} {||g||}_{s'}.
    \end{align*}
This inequality is easily derived from the $L^s$-boundedness of the $L^r$-maximal operator $M_r$. In section \ref{sec:proof} we shall also make use of the following useful fact.
\begin{Lem}\label{universal sparse form}
    Let $1\leq p,q\leq\infty$. For bounded, compactly supported functions $f,g$, there exists a sparse form ${\Lambda}_{\mathcal{S}_0 , q,p'}$ such that
    $${\Lambda}_{\mathcal{S} , q,p'}(f,g)\lesssim {\Lambda}_{\mathcal{S}_0 , q,p'}(f,g).$$
    (the hidden constant can be taken to only depend on the dimension $n$)
\end{Lem}

Obtaining $(1,1)$-sparse bounds is the best scenario, while the quite trivial $(p,p')$-sparse bounds are the weakest. \\

The following result due to J. M. Conde-Alonso, A. Culiuc, F. Di Plinio and Y. Ou \cite{domin}, shows that sparse form bounds are stronger than weak $(1,1)$ estimates.

\begin{Lem}\label{Weaksparse}
Suppose that a sublinear operator $T$ has the following property: there exists $C > 0$
and $1 \leq  q < \infty$ such that for every $f, g$ bounded with compact support there exists a sparse
collection $\mathcal{S}$ such that
$$
| \left< T f, g \right>| \leq C \Lambda_{\mathcal{S},1,q}(f,g).
$$
Then $T : L^1(\Rn) \to L^{1,\infty}(\Rn)$.
\end{Lem}

\subsection{Basic definitions and results related to the Fourier integral operators}\label{sec:fio}\quad\\

The phase functions that we shall use to define Fourier integral operators were first introduced in \cite{DS} by D. Dos Santos Ferreira and W. Staubach.

\begin{Def}\label{def:phi2}
A \textit{phase function} $\varphi(x,\xi)$ in the class $\Phi^k$, $k\in\Z_{>0}$, is a function \linebreak$\varphi(x,\xi)\in \mathcal{C}^{\infty}(\Rl^n \times\Rl^n \setminus\{0\})$, positively homogeneous of degree one in the frequency variable $\xi$ satisfying the following estimate

\begin{equation}\label{C_alpha}
	\sup_{(x,\,\xi) \in \Rl^n \times\Rl^n \setminus\{0\}}  |\xi| ^{-1+\vert \alpha\vert}\left | \partial_{\xi}^{\alpha}\partial_{x}^{\beta}\varphi(x,\xi)\right |
	\leq C_{\alpha , \beta},
	\end{equation}
	for any pair of multi-indices $\alpha$ and $\beta$, satisfying $|\alpha|+|\beta|\geq k.$
\end{Def}

\begin{Def}\label{nondeg phase}
One says that the phase function $\varphi(x,\xi)$ satisfies the strong non-degeneracy condition \emph{(}or $\varphi$ is $\mathrm{SND}$ for short\emph{)} if
\begin{equation}\label{eq:SND}
	\big |\det (\partial^{2}_{x_{j}\xi_{k}}\varphi(x,\xi)) \big |
	\geq \delta,\qquad \mbox{for  some $\delta>0$ and all $(x,\xi)\in \mathbb{R}^{n} \times \Rn\setminus\{0\}$}.
\end{equation}
\end{Def}

Now, having the definitions of the amplitudes and the phase functions at hand, we define the FIOs as follows.

\begin{Def}\label{def:FIO}
	A Fourier integral operator $T_a^\varphi$ with amplitude $a$ and phase function $\varphi$, is an operator defined \emph{(}once again a-priori on $\mathscr{S}(\Rl^n)$\emph{)} by
	\begin{equation}\label{eq:FIO}
	T_a^\varphi f(x) := \int_{\Rn} e^{i\varphi(x,\xi)}\,a(x,\xi)\,\widehat f (\xi) \dd\xi,
	\end{equation}
	where $\varphi$ is a member of $\Phi^k$ for some $k\geq 1$. We say that $\varphi$ is of rank $\kappa$ if it satisfies $\mathrm{rank}\,\partial^{2}_{\xi\xi} \varphi(x, \xi) = \kappa$ for all $(x,\xi) \in \Rn \times \Rn\setminus \{0\}$.
\end{Def}

We will refer to a Fourier integral operator with amplitude in $S^m_{\rho,\delta}$ and phase function in $\Phi^2$ as a \textit{Fourier integral operator} (or FIO for short). Taking $\varphi(x,\xi)=x\cdot\xi$, one obtains the class of pseudodifferential operators associated to H\"ormanders symbol classes.

\subsection{Basic facts related to oscillatory integral operators}\label{sec:oio}

In this section, we aim to provide a concise overview of fundamental concepts about oscillatory integral operators that will serve as the foundation for the subsequent sections of our paper. We begin by introducing the definition of an oscillatory integral operator and exploring its essential properties. \\

When dealing with oscillatory integral operators, it is crucial to focus on the phase functions, as they play a central role in their classification. We consider nonlinear phase functions that go beyond the Fourier integral operators, in particular we are interested in phase functions of the following type, originally defined in \cite{CISY}:
\begin{Def}\label{def:fk}
For $0<k<\infty$, we say that a real-valued \textit{phase function} $\varphi(x,\xi)$ belongs to the class $\textart F^k$, if
$\varphi(x,\xi)\in \mathcal{C}^{\infty}(\Rn \times\Rn \setminus\{0\})$ and satisfies the following estimates $($depending on the range of $k)$\emph:
\begin{itemize}
    \item for $k \geq 1$, 
        $$ 
|\partial^\alpha_\xi  (\varphi (x,\xi)-x\cdot\xi) |\leq
c_{\alpha} |\xi|^{k-1}, \quad  |\alpha| \geq 1 ,
$$
    \item for $0<k<1$, 
$$ 
|\partial^\alpha_\xi \partial^{\beta}_x  (\varphi (x,\xi)-x\cdot\xi) |\leq
c_{\alpha,\beta} |\xi|^{k-|\alpha|}, \quad |\alpha + \beta | \geq 1 ,
$$
\end{itemize}
for all $x\in \Rn$ and $|\xi|\geq 1$.
\end{Def}

A widely recognized and representative example of a phase in the space $\textart F^k$ is given by $|\xi|^k + x \cdot \xi$. This particular type of phase is closely associated with the operator $e^{i (\Delta)^{k/2}}$. Having the definitions of the amplitudes and the phase functions at hand, one has the following definition.
 
\begin{Def}\label{def:OIO}
    An oscillatory integral operator $T_a^\varphi$ with amplitude $a$ in $ S^{m}_{\rho, \delta}(\mathbb{R}^n)$ and a real-valued phase function $\varphi$ is defined, initially on $\mathscr{S}(\mathbb{R}^n)$, as follows:
    \begin{equation}\label{eq:OIO}
    T_a^\varphi f(x) := \int_{\mathbb{R}^n} e^{i\varphi(x,\xi)} a(x,\xi) \widehat{f}(\xi) \dd\xi,
    \end{equation}
    where $\varphi$ is a member of $\textart{F}^k$ for some $k\geq 1$. We call $k$ the order of the oscillatory integral operator.
\end{Def}

In order to ensure the global $L^2$-boundedness of our operators, an additional condition must be imposed on the phase, which we will henceforth refer to as the $L^2$-condition for the sake of simplicity. This condition plays a crucial role in controlling the $L^2$ behavior of the oscillatory integral operators under consideration.

\begin{Def}\label{L2 conditions}
One says that the phase function $\varphi(x,\xi)\in \mathcal{C}^{\infty}(\Rn \times\Rn)$ satisfies 
the $L^2$-condition if 
\begin{equation}\label{eq:L2 condition_old}
            |\partial_{x}^{\alpha} \partial_{\xi}^{\beta} \varphi(x,\xi )| \leq C_{\alpha},
\end{equation}
for $|\alpha|\geq 1$, $|\beta|\geq 1,$ all $x \in\mathbb{R}^{n} $ and $|\xi|\geq 1$.
\end{Def}

The following $\mathrm{LF}(\mu)$-condition is an inherent requirement that, from the perspective of applications to PDEs, is typically satisfied and does not introduce any loss of generality. This condition is essential for analyzing the low-frequency components of the operators.

\begin{Def}\label{Def:LFmu}
Let $\varphi(x,\xi)\in \mathcal{C}^{\infty}(\mathbb{R}^n \times \mathbb{R}^n \setminus {0})$ be a real-valued function, and consider $0<\mu\leq 1$. We define the low-frequency phase condition of order $\mu$, abbreviated as the $\mathrm{LF}(\mu)$-condition, as follows: We say that $\varphi$ satisfies the $\mathrm{LF}(\mu)$-condition if it satisfies the inequality
\begin{equation}\label{eq:LFmu}
|\partial^{\alpha}_{\xi}\partial_{x}^{\beta} (\varphi(x,\xi)-x\cdot \xi) |\leq c_{\alpha} |\xi|^{\mu-|\alpha|},
\end{equation}
for all $x\in \mathbb{R}^n$, $0<|\xi| \leq 2$, and all multi-indices $\alpha$ and $\beta$.
\end{Def}

\section{Consequences of the main results}\label{Sec:weightedConsequences}

    Sparse domination offers a notable advantage over classical $L^p$ estimates by yielding weighted norm inequalities. In the subsequent subsections, we will delve into the discussion of several of these weighted estimates, as well as implications to weak $(1,1)$ estimates and vector-valued inequalities. 

\subsection{Weighted norm inequalities}\label{sec_weight}

In this section we derive bounds for $T_a^\varphi$ of the form
\begin{equation}\label{weighted_boundedness}
    \|T_a^\varphi f\|_{L^p_w}\leq C\|f\|_{L^p_w}
\end{equation}
with $w\in A_q$, for some $1<q,p<\infty$, and where 
\begin{equation*}
    \|f\|_{L^p_w}:=\Big(\int_{\Rl^n} |f(x)|^p\,w(x)\dd x\Big)^{1/p}.
\end{equation*}
For weighted norm inequalities for FIOs the reader is referred to the work D. Dos Santos Ferreira and W. Staubach \cite{DS}. For weighted norm inequalities for oscillatory integral operators that go beyond FIOs the reader is referred to A. Bergfeldt, S. Rodriguez Lopez, and W. Staubach \cite{BS}. Note that in these papers the authors consider weighted norm inequalities for weights that belong to all $A_p$-classes, for $1<p<\infty$ (see equation \eqref{weighted_boundedness}), i.e. $q=p$ in equation \eqref{weighted_boundedness}. In both cases they disregard the quantitative control on $C$.\\

Our corollary recovers all the previous known weighted estimates of Beltran and Cladek in \cite{BC}. In the case of FIOs our result extends the weighted norm estimates in \cite{DS} and \cite{BS} to phase functions with arbitrary rank $0\leq \kappa\leq n-1$ (see item (2) below). In both the cases of FIOs and OIOs we obtain control of $C$, moreover we obtain new weighted norm inequalities for weights in the more general class $ A_{q/q_0} \cap RH_{(p_0/q)'}$, i.e. we extend to the case when $p$ and $q$ are not necessarily equal.\\

The following corollary derives from \cite[Proposition 6.4]{Bernicot1} (due to F. Bernicot, D. Frey and S. Petermichl) and Theorem \ref{main_FIO} in the case of FIOs and Theorem \ref{main_OIO} in the case of OIOs.

\begin{Cor}\label{weightedConsequencesSmooth} 
Let $a\in S^{m}_{\rho,\delta}$ be an amplitude with $m\leq 0$ and $0\leq\rho\leq 1$, $0\leq\delta< 1$. Let the phase functions $\varphi$ be SND and satisfy either 
    \begin{enumerate}
        \item $\varphi\in \Phi^2$, or
        \item $\varphi\in \textart{F}^k$, obeys the $L^2$-condition and is $LF(\mu)$ for some $0< \mu\leq  1$.
    \end{enumerate}
If (1) holds, take $m$ and all pairs $(q_0, p_0')$, $(p_0', q_0)$ such that \eqref{FIO_ranges} is satisfied.
While if (2) holds,  take $(q_0, p_0')$, $(p_0', q_0)$ and $m$ such that \eqref{OIO_ranges} is satisfied. For $w \in A_{q/q_0} \cap RH_{(p_0/q)'}$ and all $q_0<p<p_0$, there exists a constant $c_{p}$
\begin{equation*}
    \|T_a^\varphi f\|_{L^p_w\to L^p_w}\leq c_p\Big([w]_{A_{q/q_0}}[w]_{RH_{(p_0/q)'}}\Big)^{\alpha}
\end{equation*}
where
\begin{equation*}
    \alpha=\max\Big\{\frac{1}{p-q_0},\frac{p_0-1}{p_0-p}\Big\}.
\end{equation*}
\end{Cor}

We define the class dependent number $D$ as follows. Let $$D=D(n,\varkappa)=n(1-\varkappa)$$ if $T_a^\phi$ is an OIO, as in Theorem \ref{main_OIO}, and let $$D=D(n,\rho)=n(1-\rho)$$ if $T_a^\phi$ is an FIO, as in Theorem \ref{main_FIO}. 
\begin{Rem}\label{weighted_remark}
    Observe that using the \textit{open property} of $A_p$ weights (see \cite{hytonen} or \cite{Stein}) and that $A_p$ classes increase in $p$, one can show the following two cases of corollary \ref{weightedConsequencesSmooth} for OIOs. For more details on proving this we refer the reader to \cite{BC}.
    \begin{enumerate}
        \item Let $-D-\zeta<m\leq -D/2-\zeta$. Then $T_a^\varphi$ is bounded on $L^p_w$ for $w \in A_{-\frac{p(m+\zeta)}{D}}$ and all $-\frac{D}{m+\zeta}<p<\infty$.  
        \item Let $m=-D/2-\zeta$. Then $T_a^\varphi$ is bounded on $L^p_w$ for $w \in A_{p/2}$ and all $2<p<\infty$.
    \end{enumerate}
    These results only constitute some of the weighted estimates implied by corollary \ref{weightedConsequencesSmooth}. 
\end{Rem}

\subsection{Weak-type estimates}

For weighted norm inequalities for local FIOs the reader is referred to the work by T. Tao \cite{Weak_tao}. Note that in that paper the author consider FIOs with amplitudes in $S^m_{1,0}$ ($m=-(n-1)/2$) and compact support in the spatial variable $x$, as well as phase functions which are smooth away from $|\xi|=0$, satisfy the SND condition, and are homogeneous of degree $1$. Our corollary extends this result to global FIOs, with phase functions that are in $\Phi^2$ and amplitudes in the $S^m_{\rho,\delta}$ for all $0\leq\rho\leq 1$ and $0\leq\delta<1$ up to the end point. In particular when the phase has full rank, we obtain $m<-(n-1)/2$ for $a\in S^m_{1,0}$, recovering the result in \cite{Weak_tao} up to the end point.\\

For global oscillatory integral operators our corollary is entirely new, in particular the numerology in corollary \ref{weak_oio} agrees with the $L^p$ estimates obtained in \cite{IMS2} by A. Israelsson, W. Staubach and the author. \\

Combining Lemma \ref{Weaksparse} with Theorem \ref{main_FIO} and Theorem \ref{main_OIO} respectively, we obtain the following two corollaries.

\begin{Cor}\label{weak_fio}
    Assume that $a(x,\xi)\in S^{m}_{\rho,\delta}(\Rn)$ for $0\leq \rho\leq 1$ and $0\leq \delta< 1$, and $\varphi(x,\xi)$ is in the class $\Phi^2$ with rank $0\leq\kappa\leq n-1$ and is \emph{SND}. Then, for
    \begin{equation*}
        m<
        \begin{cases}
            -n(1-\rho)-\zeta, & 0\leq\rho< \frac{n}{n+\kappa},\\
            -\frac{n(1-\rho)+\kappa\rho}{2}-\zeta, &  \frac{n}{n+\kappa}\leq\rho\leq 1,
        \end{cases}
    \end{equation*}
    it holds that
    $$\|T_a^\varphi f\|_{L^{1,\infty}(\Rn)}\lesssim \|f\|_{L^{1}(\Rn)}.$$
\end{Cor}

\begin{Cor}\label{weak_oio}
    Let $n\geq 1$, $0<k<\infty$. Assume that $\varphi\in \textart F^k$ {is \emph{SND}}, satisfies the \emph{LF}$(\mu)$-condition for some $0<\mu\leq 1$, and the $L^2$-condition \eqref{eq:L2 condition_old}. Assume also that $a(x,\xi)\in S^{m}_{\rho,\delta}(\Rn),$ for $0\leq \rho\leq 1$ and $0\leq \delta< 1$. Then, for $m<-\frac{n(1-\varkappa)}{2}-\zeta$ it holds that
    $$\|T_a^\varphi f\|_{L^{1,\infty}(\Rn)}\lesssim \|f\|_{L^{1}(\Rn)}.$$
\end{Cor}

\subsection{Vector-Valued inequalities}

By leveraging the $A_p$ weighted estimates obtained earlier and Lemma \ref{linearizable}, we can derive new vector-valued estimates for OIOs and FIOs. Define $\|\cdot\|_{L^q(\ell^p)}$ to be the quasi-norm
\begin{align*}
    \|(f_k)\|_{L^q(\ell^p)}&:=\Big\|\sum_{k} |f_k(\cdot)|^p\Big\|_{L^q(\Rn)}^{1/p}.
\end{align*}
Define the space $L^q(\ell^p)$ as the sequence space in which $\|(f_k)\|_{L^q(\ell^p)}<\infty$. 

As we mentioned in the introduction there are important applications for FIOs and OIOs with simplified phase functions of the form $x\cdot\xi+\phi(\xi)$. The following result is a weighted vector-valued estimate for FIOs and OIOs with simple phase functions. For weighted Triebel-Lizorkin estimates we refer the reader to \cite{DS}. Since the proof of the result below follows essentially in the same way as theirs, we omit the details and refer the reader to that paper.

\begin{Th}\label{Vector-valued_smooth}
Let $a\in S^{m}_{\rho,\delta}$be an amplitude with $m\leq 0$ and $0\leq\rho\leq 1$, $0\leq\delta< 1$. Let the phase functions $\varphi$ be SND and satisfy either 
    \begin{enumerate}
        \item $\varphi\in \Phi^2$, or
        \item $\varphi\in \textart{F}^k$, and $\varphi$ is $LF(\mu)$ for some $0< \mu\leq  1$.
    \end{enumerate}

    If (1) holds, take $m=-n(1-\rho)-\zeta$, while if (2) holds take $m=-n(1-\varkappa)-\zeta$. For all $1<q,p<\infty$, and $(f_j)_{j\geq 0}\in L^q(\ell^p)$ it holds that
    \begin{align}\label{Rubioeq}
        \|\Big(\sum_{j=0}^\infty |T_{a}^{\varphi} f_j|^p\Big)^{1/p}\|_{L^q_w(\Rl^n)}\lesssim \|\Big(\sum_{j=0}^\infty |f_j|^p\Big)^{1/p}\|_{L^q_w(\Rl^n)}.
    \end{align}
    Moreover, for $s'\geq s$ it holds that
    \begin{align}\label{Rubioeq1}
        \|T_{a}^{\varphi}\|_{F^{s'}_{p,q}(w)\to F^s_{p,q}(w)}\lesssim 1,
    \end{align}
    whenever the phase is simple. \emph{(}See Definition \ref{def:TLspace} for the details on the space $F^s_{p,q}(w)$.\emph{)}
\end{Th}

In the above result \eqref{Rubioeq} is a direct consequence of Theorem \ref{Rubio}, and item (2) in remark \ref{weighted_remark} applied to $T_{a}^{\varphi}$. For the reader interested in the proof of estimate \eqref{Rubioeq1} we refer to \cite{DS}. The argument is rather short and is essentially an application of a variant of \eqref{Rubioeq}, and using that $[\psi_j,T_{a}^{\varphi}]=0$ for simple phase functions. 

\section{Decompositions and Kernel Estimates}\label{sec:decomp}

In this section we provide the necessary tools used in the proofs of Theorem \ref{main_FIO} and Theorem \ref{main_OIO}.

\subsection*{Decomposition of Fourier integral operators}

In this section we introduce the spatial and frequency decompositions related to the FIOs, we refer the reader to \cite{IMS} for the origin of this decomposition and details on the proofs of the Lemma \ref{lem:chijnu} and Lemma \ref{lem:h}.\\

To start, one considers an FIO $T^{\varphi}_{a}$ with amplitude $a(x, \xi)\in S^{m}_{\rho, \delta}(\Rl^n)$, $0<\rho\leq 1$, $0\leq \delta<1$, $m\in\Rl$ and $0 \leq \kappa\leq n-1$, $\varphi\in\Phi^2$,  $\mathrm{rank}\,\partial^{2}_{\xi\xi} \varphi(x, \xi) = \kappa$, on the support of $a(x, \xi)$ and its Littlewood-Paley decomposition
\begin{equation}\label{eq:LPdecomp}
T^{\varphi}_{a}= \sum_{j=0}^\infty  T^{\varphi}_{a}\psi_j(D) =: \sum_{j=0}^\infty  T_j, 
\end{equation}
where the kernel $K_j$ of $T_j$ is given by
\begin{align*}
    &K_j(x,y):= \int_{\Rn} e^{i\varphi(x,\xi)-iy\cdot\xi}\,\psi_j(\xi)\,a(x,\xi)\, \dd \xi.
\end{align*}
Here each $\psi_j$ (for $j\geq 1$) is supported in a dyadic shell $\left \{2^{j-1}\leq \vert \xi\vert\leq 2^{j+1}\right \}$.\\

The shells $A_j$ will in turn be decomposed into truncated cones using the following construction. Since $\varphi$ has constant rank $\kappa$ on the support of $a(x, \xi)$ we may assume that there exists some $\kappa$-dimensional submanifold $S_\kappa(x)$ of $\mathbb S^{n-1}\cap \Gamma$ for some sufficiently narrow cone $\Gamma$, such that $\mathbb S^{n-1}\cap \Gamma$ is parameterised by $\overline{\xi}=\overline{\xi}_x(u,v)$, for $(u,v)$ in a bounded open set $U\times V$ near $(0,0)\in \Rl^{\kappa}\times\Rl^{n-\kappa-1}$, and such that $\overline{\xi}_x(u,v)\in S_\kappa(x)$ if and only if $v=0$, and $\nabla_\xi\varphi(x,\overline{\xi}_x(u,v))=\nabla_\xi\varphi(x,\overline{\xi}_x(u,0))$.

\begin{Def}\label{def:LP2}
For each $j\in\mathbb Z_{>0}$, let $\{u^{\nu}_{j} \} $ be a collection of points in $U$ such that
\begin{enumerate}
    \item [$(i)$]$  | u^{\nu}_{j}-u^{\nu'}_{j} |\geq 2^{-j/2},$ whenever $\nu\neq \nu '$.
    \item [$(ii)$] If $u\in U$, then there exists a $
    u^{\nu}_{j}$ so that $\vert u -u^{\nu}_{j}  \vert
    <2^{-j/2}$.
\end{enumerate}
Moreover we set $\xi^{\nu}_{j}=\overline{\xi}_x(u^{\nu}_{j},0).$ 
One may take such a sequence by choosing a maximal collection $\{\xi_{j}^{\nu}\}$ for which $(i)$ holds, then $(ii)$ follows. Now, denote the number of cones needed by $\mathscr N_j$.
\end{Def}

Let $\Gamma^{\nu}_{j}$ denote the cone in the $\xi$-space given by
\begin{equation}\label{eq:gammajnu}
\Gamma^{\nu}_{j}
:= \{ \xi \in \Rl^n ;\, \xi=s\overline{\xi}_x(u,v), | \frac{u}{|u|} - u^{\nu}_{j} | <2^{-j/2}, v\in V, s>0\}.
\end{equation}
We decompose each $A_j$ into truncated cones $\Xi^{\nu}_{j}:=\Gamma^{\nu}_{j}\cap A_j.$ One also defines the partition of unity $\{\chi_j^\nu(u)\}_{j,\nu}$ on $\Gamma$ (a construction can be found in \cite{IMS}). This homogeneous partition of unity $\chi_j^\nu$ is defined such that $\chi_j^\nu (s\overline{\xi}_x(u,v))=\tilde\chi_j^\nu (u),$
for $v\in V$ and $s>0$. Using this decomposition, we make a second dyadic decomposition of $T_j$ as 
\begin{equation*}
    T_j^\nu f(x) := \int_{\Rn}K_j^\nu(x,y)\, f(y)\dd y,
\end{equation*}
where
\begin{align*}
    K_j^\nu(x,y):= \int_{\Rn} e^{i\varphi(x,\xi)-iy\cdot\xi}\,\psi_j(\xi)\,\chi_j^\nu(\xi)\,a(x,\xi)\dd \xi .
\end{align*}
Observe that $T_j = \sum_{\nu=1}^{\mathscr N_j} T_j^\nu.$ We choose the coordinate axes in $\xi$-space such that {$\xi''\in (T_{\xi^{\nu}_j}S_\kappa(x))^\perp$ and $\xi'\in T_{\xi^{\nu}_j}S_\kappa(x)$}, so that $\xi=(\xi'',\xi')$. 

\begin{Lem}\label{lem:chijnu}
    The functions $\chi_j^\nu$ belong to $\mathcal C^\infty(\Rl^n\setminus \{0\} )$ and are supported in the cones $\Gamma_j^\nu$. They satisfy 
    \begin{align*}
        |\d_\xi^\alpha \chi_j^\nu (\xi) |\lesssim 
    2^{j|\alpha|/2}
    |\xi|^{-|\alpha|}
    \end{align*}
    for all multi-indices $\alpha$.
\end{Lem}

We will also make us of the following slightly modified version of the partition of unity
\begin{equation}\label{chialt}
    \tilde{\chi}_j^\nu := \eta_j^\nu \Big(\sum_\nu \big(\eta_j^\nu\big)^2\Big)^{-\frac{1}{2}},
\end{equation}
for which
\begin{equation}
    \sum_\nu \tilde{\chi}_j^\nu(\xi)^2 = 1,\quad \text{ for all } j\text { and } \xi\neq 0.
\end{equation}
This partition of unity also satisfies the estimates in Lemma \ref{lem:chijnu}. In connection to this partition of unity we also define the fattened version of the second dyadic decomposition. We define these pieces by $\mathcal{X}_j^{\nu}(D):=\tilde{\chi}_j^\nu(D)\Psi_j(D)$. Observe that these have the nice property that they satisfy the equation
\begin{equation}\label{fatsecond}
    T_j^\nu f(x)=T_j^\nu f_j^\nu (x),
\end{equation}
where $f_j^\nu (x):=\mathcal{X}_j^{\nu}(D) \Psi_j(D) f(x)$. Moreover, these pieces satisfy the following estimate
\begin{equation} \label{eq:high}
	\|(\mathcal X_{j}^{\nu})^\vee\|_{L^p(\Rl^n)} \lesssim  2^{j\left (\frac{n+1}{2}-\frac{n+1}{2p}\right )},
\end{equation}
the proof of which follows using that $\tilde{\chi}_j^\nu(D)$ satisfies lemma \ref{lem:chijnu} and partial integration.\\

We will split the phase $\varphi(x,\xi)-y\cdot\xi$ into two different pieces,  $(\nabla_\xi\varphi(x,\xi_j^\nu)-y)\cdot \xi$ (which is linear in $\xi$), and $\varphi(x,\xi) - \nabla_\xi \varphi(x,\xi_j^\nu) \cdot \xi$. The following lemma yields an estimate for the nonlinear second piece. For $j\in \Z_{\geq0}$, $1\leq\nu \leq \mathscr N_j,$ define
$$h_j^\nu(x,\xi):= \varphi(x,\xi) -\nabla_\xi \varphi(x,\xi_j^\nu)\cdot \xi.$$

\begin{Lem}\label{lem:h}
Let $\varphi(x,\xi)\in\mathcal{C}^{\infty}(\Rl^n \times \Rl^n\setminus\{0\})$ be positively homogeneous of degree one in $\xi$. 
Then for $\xi$ in $\Xi^{\nu}_{j}$, one has for $\alpha\in \mathbb Z_{> 0}^{\kappa}$ that 
$$
    |{\partial_{\xi'}^\alpha} h_j^\nu(x,\xi)| \lesssim \begin{cases}
    |\xi'| |\xi|^{-1} & \text{ if } |\alpha|= 1  \\ 
    |\xi| ^{1-|\alpha|} & \text{ if } |\alpha| \geq 2
    \end{cases} \,\lesssim 2^{-j/2}.
    $$
Moreover, for $\beta\in \mathbb Z_{> 0}^{n-\kappa}$ one has that
$$
    |\partial_{\xi''}^\beta h_j^\nu(x,\xi)|\lesssim |\xi'|^2 |\xi| ^{-|\beta|-1}\lesssim 2^{-j}.
$$
And finally one has the estimate
    \begin{equation}\label{eq:trunkerad0}
        |\nabla_{\xi} h^{\nu}_j(x,\xi)|\lesssim 2^{-j/2},\quad \xi\in  \Xi^{\nu}_j.
    \end{equation}
\end{Lem}

Define $\wt{\psi}_{j,l}^{\nu}(x,y)$ as the characteristic function of the set 
\begin{equation*}
    \{(x,y)\in\Rl^{2n}:
    2^l\leq (1+2^{2j}|\nabla_{\xi''} \varphi(x,\xi_j^\nu)-y|^2)(1+2^{j}|\nabla_{\xi'} \varphi(x,\xi_j^\nu)-y|^2)< 2^{l+1}\}.
\end{equation*}

We further decompose each $T_j$ into spatially localised operators $T_{j, l}^\nu$ defined by
\begin{align*}
T_{j, l}^\nu f(x)&=\int_{\Rn}f(y)\wt{\psi}_{j,l}^{\nu}(x,y) K_j^\nu(x,y)\dd y
\end{align*}
for $l \in \Z$ and $j>0$, and in an analogous manner for $T_{0,l}$, so that
\begin{align*}
T_a^\varphi=\sum_{j\ge 0}T_j=\sum_{l\in\mathbb{Z}}\sum_{j\ge 0}\sum_{\nu}T_{j, l}^\nu.
\end{align*}

We further group the pieces $T_{j,l}$ according to their spatial scale. Fix some $\epsilon>0$ and write
\begin{align*}
T_j^\nu=\sum_{l\le j\epsilon}T_{j, l}^\nu+\sum_{l>j\epsilon}T_{j, l}^\nu.
\end{align*}

In the next lemma we estimate the regularity of the kernel $$K_{j,l}^\nu(x,y):=\wt{\psi}_{j,l}^{\nu}(x,y) K_{j}^\nu(x,y)$$ of $T_{j,l}^\nu$.

\begin{Lem}\label{kernel lemma_FIO}
        Let $0\leq \rho,\delta\leq 1$, $\rho\neq 0$, $n\geq 1$. Assume that $\varphi\in \Phi^2$.  Then if $a(x,\xi)\in S^{m}_{\rho,\delta}(\Rn)$, we have for $1\leq p\leq\infty$ that
        \begin{equation}\label{kernel estimate_FIO}
            \|w_j^{\nu}(x,y)^{N} K_{j,l}^{\nu} (x,y)\|_{L^{p}_y(\Rn)}  \lesssim 2^{j(m+n/p'-1/p)} 2^{-lL},
        \end{equation}
        for all $j\geq 0$ and all $N,L\geq 0$, and where
        \begin{equation}\label{metric definition}
            w_j^{\nu}(x,y)=(1+2^{2j}|\nabla_{\xi''} \varphi(x,\xi_j^\nu)-y''|^2)(1+2^{j\rho}|\nabla_{\xi'} \varphi(x,\xi_j^\nu)-y'|^2).
        \end{equation}
\end{Lem}

\begin{proof}
For any $j\geq 0$ we have
\begin{align*}
K_{j,l}^\nu(x,y) 
& = \wt{\psi}_{j,l}^{\nu}(x,y) \int_{\Rn}  \psi_j(\xi)\,\chi_j^\nu(\xi)\,
e^{i(\varphi(x,\xi )-y\cdot \xi)} \,a(x,\xi) \dd \xi \\
& = \wt{\psi}_{j,l}^{\nu}(x,y) \int_{\Rn}
e^{i(\varphi(x,\xi )-y\cdot \xi)}\, \sigma_j^{\nu}(x,\xi) \dd\xi,
\end{align*}
where
$$\sigma_{j}^{\nu}(x,\xi)
:=  \psi_j(\xi)\,\chi_j^\nu(\xi) \, a(x,\xi).
$$
Observe that by Lemma \ref{lem:h} we have
    \begin{align*}
        |\d_\xi^\alpha \chi_j^\nu (\xi) |\lesssim 
    2^{j(\rho/2-1)|\alpha|},
    \end{align*}
on the support of $\psi_j$. Using this and that $a(x,\xi)\in S^{m}_{\rho,\delta}(\Rn)$ we can calculate that for any multi-index $\gamma$, any $j\geq 0$ and any $\nu$ one has
\begin{equation}\label{amplitude derivative estim_FIO}
 |\partial^{\gamma}_\xi \sigma_j^{\nu}(x,\xi)|\lesssim 2^{j(m-\min\{\rho, 1-\rho/2\}|\gamma|)}\lesssim 2^{jm}.
\end{equation}

If we now set $h_j^\nu(x,\xi):=\varphi(x,\xi)-\xi \cdot \nabla_\xi \varphi(x,\xi_j^\nu)$, then we can write
\begin{align*}
     K_{j,l}^\nu(x,y) &= \wt{\psi}_{j,l}^{\nu}(x,y)
     \int_{\Rn}  e^{i(\nabla_\xi\varphi(x,\xi_j^\nu )-y)\cdot \xi}\,e^{ih_j^\nu(x,\xi)}\,\sigma_{j}^{\nu}(x,\xi)  \dd \xi.
\end{align*}
 Now we estimate the derivatives of $h_j^\nu$ in $\xi$ on the support of $\sigma_{j}^{\nu}(x,\xi)$. To this end, Lemma \ref{lem:h} yields
\begin{equation}
        |\nabla_{\xi} h^{\nu}_j(x,\xi)|\lesssim 2^{-j\rho /2}\lesssim 1,\quad \xi\in  \Xi^{\nu}_j,
\end{equation}
and
\begin{align}\label{nonlinear_FIO}
|\partial_{\xi}^{\alpha}\vartheta_j^\nu(x,\xi)| = |\partial_{\xi}^{\alpha}\varphi(x,\xi)|\lesssim 1,\quad \text{ for all } |\alpha| \geq 2.
\end{align}

Next we introduce the operator $L$ defined by
\begin{align*}
    L=(I-2^{2j}\Delta_{\xi''})(I-2^{j\rho}\Delta_{\xi'}).
\end{align*}

Then using the estimate \eqref{amplitude derivative estim_FIO} and \eqref{nonlinear_FIO} we obtain, 
\begin{align*}
    L^N(e^{i\vartheta_j^\nu(x,\xi)}\,\sigma_{j}^{\nu}(x,\xi))\lesssim 2^{jm}
\end{align*}
This estimate can be improved in terms of obtaining a decaying factor depending on $\rho$, however for our purposes this is not necessary since we are looking for geometric decay, which will be achieved in a different manner below. We also have that
\begin{align*}
    e^{i(\nabla_\xi\varphi(x,\xi_j^\nu )-y)\cdot \xi}=\frac{L^N(e^{i(\nabla_\xi\varphi(x,\xi_j^\nu )-y)\cdot \xi})}{w_j^{\nu}(x,y)^N}\, 
\end{align*}
Therefore, partial integration and the trivial estimate $|\Xi^{\nu}_{j}|\lesssim 2^{jn}$ yields
\begin{align*}
   |K_{j,l}^\nu(x,y) |
    &\lesssim\frac { \wt{\psi}_{j,l}^{\nu}(x,y) 2^{j(m+n)}}{w_j^{\nu}(x,y)^N}.
\end{align*}

The observation that $$w_j^{\nu}(x,y)^N\gtrsim 2^{lN}$$ on the support of $\widetilde{\psi}_{j,l}^\nu$ gives us for any $N,L\geq 0$ and $1\leq p\leq\infty$ that
\begin{align*}
    &\|w_j^{\nu}(x,y)^{N} K_{j,l}^{\nu} (x,y)\|_{L^{p}_y(\Rn)}2^{-j(m+n)}\\
    &= \Big(\int_{\Rn} \frac{\wt{\psi}_{j,l}^{\nu}(x,y)w_j^{\nu}(x,y)^{-pL}}{w_j^{\nu}(x,y)^{-p(N+L-M)}}\dd y \Big)^{1/p}\\
    &\lesssim 2^{-lL}\Big(\int_{\Rn} w_j^{\nu}(x,y)^{p(N+L-M)}\dd y \Big)^{1/p}\\
    &\lesssim 2^{-lL}2^{-j(n+1)/p},
\end{align*}
for $M>n+N+L$.
Therefore we obtain, 
\begin{align*}
    \|w_j^{\nu}(x,y)^{N} K_{j,l}^{\nu} (x,y)\|_{L^{p}_y(\Rn)}
    \lesssim 2^{j(m+n/p'-1/p)} 2^{-lL}.
\end{align*}
\end{proof}

\subsection*{Decomposition of oscillatory integral operators}

In this section we introduce the spatial and frequency decompositions related to the OIOs, we refer the reader to \cite{IMS2} for the origin of this decomposition. It is an adaption of the classical second dyadic decomposition introduced by C. Fefferman, to the case of nonlinear phases of the type $\textart {F}^k$.\\

Let $k>0$.
    We make the following decomposition of the integral kernel 
$$K(x,y)= \int_{\Rn}  a(x,\xi)\,e^{i\varphi(x,\xi)-iy\cdot \xi}\,  \dd \xi.$$

Then for every $j\geq 0$ we cover  
$\supp \psi_j$ with open balls $C_j^\nu$ with radii $2^{j(1-k)}$ and centres $\xi_j^\nu$, where $\nu$ runs from $1$ to $\mathscr{N}_j:=O(2^{njk})$.  Observe that $|C_j^\nu|\lesssim 2^{n j(1-k)}$, uniformly in $\nu$. 
Now take $u\in \mathcal{C}_c^{\infty}(\Rn)$, with $0\leq u\leq 1$ and supported in $B(0,2)$ with $u=1$ on $\overline{B(0,1)}.$ Next set
$$\chi_j^\nu(\xi) := \frac{u(2^{-(1-k) j}(\xi -\xi_{j}^{\nu}))} {\sum_{\kappa=1}^{\mathscr{N}_j} u(2^{-(1-k) j}(\xi-\xi_{j}^{\kappa}))},$$
and note that
$$\sum_{j=0}^\infty \sum_{\nu=1}^{\mathscr N_j} \chi_j^\nu(\xi)\,\psi_j(\xi) =1.$$
Now we define the second order frequency localized pieces of the kernel above as
\begin{equation*}
K_{j}^\nu (x,y) := \int_{\Rn}  \psi_j(\xi)\,\chi_j^\nu(\xi)\,e^{i\varphi(x,\xi)-iy\cdot \xi}\,a(x,\xi)  \dd \xi,
\end{equation*}

and define $$\wt{\psi}_{j,l}^{\nu}(x,y):=\wt{\psi}(2^{-l+j\varkappa}(\nabla_\xi \varphi(x,\xi_j^\nu)-y)).$$ 
We further decompose each $T_j$ into spatially localised operators $T_{j, l}^\nu$ defined by
\begin{align*}
T_{j, l}^\nu f(x)&=\int_{\Rn}f(y)\wt{\psi}_{j,l}^{\nu}(x,y) K_j^\nu(x,y)\dd y
\end{align*}
for $l \in \Z$ and $j>0$, and in an analogous manner for $T_{0,l}$, so that
\begin{align*}
T_a^\varphi=\sum_{j\ge 0}T_j=\sum_{l\in\mathbb{Z}}\sum_{j\ge 0}\sum_{\nu}T_{j, l}^\nu.
\end{align*}

We further group the pieces $T_{j,l}$ according to their spatial scale. Fix some $\epsilon>0$ and write
\begin{align*}
T_j^\nu=\sum_{l\le j\epsilon}T_{j, l}^\nu+\sum_{l>j\epsilon}T_{j, l}^\nu.
\end{align*}

In the next lemma we estimate the kernel $$K_{j,l}^\nu(x,y):=\wt{\psi}_{j,l}^{\nu}(x,y) K_{j}^\nu(x,y)$$ of $T_{j,l}^\nu$, to this end define the notation
$$w(k,\rho)=
\begin{cases}
    \min\{\rho,1-k\} & 0<k<1,\\
    k-1 & k\geq1.\\
\end{cases}$$

\begin{Lem}\label{kernel lemma}
        Let $0\leq \rho,\delta\leq 1$, $n\geq 1$, and $0<k<\infty $. Assume that $\varphi\in \textart F^k$.  Then if $a(x,\xi)\in S^{m}_{\rho,\delta}(\Rn)$, we have for $1\leq p\leq\infty$ that
        \begin{equation}\label{kernel estimate}
            \|(1+2^{jw(k,\rho)}|x-y|)^{N} K_{j,l}^{\nu} (x,y)\|_{L^{p}_y(\Rn)}  \lesssim 2^{-lL+j\varkappa L}2^{-jw(k,\rho)(n/p+L)},
        \end{equation}
        for all $j\geq 0$ and all $N,L\geq 0$.
\end{Lem}

\begin{proof}
From Lemma 3.2 in \cite{IMS2} we obtain the following inequality
\begin{equation*}
| K_{j,l}^\nu (x,y)|
\lesssim \wt{\psi}_{j,l}^{\nu}(x,y)\frac {2^{jm} 2^{jn(1-k)}}{\jap{2^{jw(k,\rho)}(\nabla_\xi \varphi(x,\xi_j^\nu)-y)}^{M}},
\end{equation*}
for all $j\geq 0$ and $M\geq 0$. The observation that $$|\nabla_\xi \varphi(x,\xi_j^\nu)-y|^{-M}\lesssim 2^{-lM+j\varkappa M}$$ on the support of $\wt{\psi}_{j,l}^{\nu}$ gives us for any $N,L\geq 0$ and $1\leq p\leq\infty$ that
\begin{align*}
    &\|(1+2^{jw(k,\rho)}|\nabla_\xi \varphi(x,\xi_j^\nu)-y|)^{N} K_{j,l}^{\nu} (x,y)\|_{L^{p}_y(\Rn)}\\
    &= \Big(\int_{\Rn} \frac{\wt{\psi}_{j,l}^{\nu}(x,y)(1+2^{jw(k,\rho)}|\nabla_\xi \varphi(x,\xi_j^\nu)-y|)^{-pL}}{(1+2^{jw(k,\rho)}|\nabla_\xi \varphi(x,\xi_j^\nu)-y|)^{-p(N+L-M)}}\dd y \Big)^{1/p}\\
    &\lesssim 2^{-lL+j\varkappa L}2^{-jw(k,\rho)L}\Big(\int_{\Rn} (1+2^{jw(k,\rho)}|\nabla_\xi \varphi(x,\xi_j^\nu)-y|)^{p(N+L-M)}\dd y \Big)^{1/p}\\
    &\lesssim 2^{-lL+j\varkappa L}2^{-jw(k,\rho)(n/p+L)},
\end{align*}
for $M>n+N+L$. 
\end{proof}

\subsection{A low frequency sparse domination result for Fourier and oscillatory integral operators}

In this section we prove a low frequency sparse form bound, which yields the necessary boundedness for all the low frequency parts considered in this paper. Therefore, from now we only consider the high frequency portions of the operators at hand.

\begin{Th}\label{low-freq-sparse}
    Let $a\in S^m_{\rho,\delta}$ with $m\in \Rl$ and $0\leq\rho,\delta\leq 1$ and let the phase function $\varphi$ be SND and satisfy either 
    \begin{enumerate}
        \item $\varphi\in \Phi^2$ or
        \item $\varphi\in \textart{F}^k$, obeys the $L^2$-condition and $\varphi$ is $LF(\mu)$ for some $0< \mu\leq  1$.
    \end{enumerate}
     Let $\chi\in C_c^\infty$ compactly supported around the origin, and define the OIO $$T_\chi f(x):=\int_{\Rn}  \chi(\xi) a(x,\xi)\,e^{i\varphi(x,\xi)-iy\cdot \xi}\, \widehat{f}(\xi) \dd \xi.$$ Then for any compactly supported bounded functions $f, g$ on $\mathbb{R}^n$, there exist sparse collections $\mathcal{S}$ and $\mathcal{S}'$ such that
        \begin{align*}
            \begin{cases}
            |\left<T_a^\varphi f, g\right>|\le C(m, q, p)\Lambda_{\mathcal{S}, q, p'}(f, g), \text{ and }\\
            |\left<T_a^\varphi f, g\right>|\le C(m, q, p)\Lambda_{\widetilde{\mathcal{S}}, p', q}(f, g),
            \end{cases}
        \end{align*}
        for all pairs $(q, p')$ and $(p', q)$ such that 
        \begin{align*}
            1\leq q\le p \le 2,\quad \text{and}\quad
            1\le q\le 2 \le p \le q'.
        \end{align*}
        \end{Th}

\begin{proof}
    We begin by considering the case of FIOs. Fix a $\xi_0\in S^{n-1}$, and define $$\wt{\psi}_{l}(x,y):=\wt{\psi}(2^{-l}(\nabla_\xi \varphi(x,\xi_0)-y)).$$ 
    We decompose $T_\chi$ into spatially localised operators $T_{\chi, l}$ defined by
    \begin{align*}
    T_{\chi, l} f(x)&=\int_{\Rn}f(y)K_l(x,y)\dd y
    \end{align*}
    with $K_l(x,y)= \wt{\psi}_{\chi,l}(x,y) K(x,y)$. Let $b(x,\xi)=e^{i(\nabla_\xi\varphi(x,\xi)-\nabla_\xi\varphi(x,\xi_0)\cdot \xi)}a(x,\xi)\chi(\xi)$, then 
    \begin{align*}
        T_\chi f(x)=\int_{\Rn}  b(x,\xi)\,e^{i(\nabla_\xi\varphi(x,\xi_0)-y)\cdot \xi}\, \widehat{f}(x) \dd \xi.
    \end{align*}
    
    Observe that for all $\mu\in[0,1)$,
    \begin{align*}
        |K_l(x,y)| &\lesssim \wt{\psi}_{l}^{\nu}(x,y) (1+|(\nabla_\xi \varphi(x,\xi_j^\nu)-y)''|^2)^{-n-\mu}(1+|(\nabla_\xi \varphi(x,\xi_j^\nu)-y)'|^2)^{-n-\mu}.
    \end{align*}
    The proof of this estimate follows from the same arguments as  Theorem 1.2.11 in \cite{DS} except one uses integration by parts with the differential operator given by $L=(I-\Delta_{\xi''})(I-\Delta_{\xi'})$ instead of just $L=\d_\xi$. The rest of the proof for FIOs follows in a very similar manner to that of Theorem \ref{main_FIO}. \\

    \textit{Handling the OIOs:} In the case of the OIOs, define $$\wt{\psi}_{l}(x,y):=\wt{\psi}(2^{-l}(x-y)).$$ 
    We decompose $T_\chi$ into spatially localised operators $T_{\chi, l}$ defined by
    \begin{align*}
    T_{\chi, l} f(x)&=\int_{\Rn}f(y)K_l(x,y)\dd y
    \end{align*}
    with $K_l(x,y)= \wt{\psi}_{\chi,l}(x,y) K(x,y)$. Observe that when $\varphi\in \textart{F}^k$ satisfies the $LF(\mu)$-condition then Lemma 4.3 in \cite{CISY} yields 
    \begin{align*}
        |K_l(x,y)| &\lesssim \wt{\psi}_{l}(x,y)  \jap{x-y}^{-(n+\eps_1\mu)},
    \end{align*}
    for any $0\leq\eps_1<1$. Thus, by continuing in the same way as was done for the FIOs above, one obtains the desired sparse form bound.
\end{proof}

\section{$L^q$-improving estimates for oscillatory integral operators}\label{sec:proof_oio}

This section is devoted to showing the $L^q\to L^p$ boundedness of OIOs. Recall that because of the low frequency sparse bound Lemma \ref{low-freq-sparse} it suffices to consider the high frequency part in this section, therefore we do not need the $LF(\mu)$-condition either, which is needed in Theorem \ref{main_OIO}.

\begin{Lem}\label{lplq_oio}
Let $n\geq 1$, $0<k<\infty$. Assume that $\varphi\in \textart F^k$ {is \emph{SND}}, and the $L^2$-condition \eqref{eq:L2 condition_old}. Assume also that $a(x,\xi)\in S^{m(p)}_{\rho,\delta}(\Rn),$ for $0\leq \rho\leq 1$ and $0\leq \delta< 1$. Then for $1\leq q\leq p\leq 2$
\begin{align*}
    \begin{cases}
        \|\sum_{l\le j\epsilon}T_{j, l}\|_{L^q\to L^p}\lesssim_{\epsilon} 2^{j(m+\frac{n(1-\varkappa)}{2}\big(\frac{2}{p}-1\big)+\zeta)}2^{-jn\big(\frac{1}{p}-\frac{1}{q}\big)},\\ 
        \|T_{j, l}\|_{L^q\to L^p}\lesssim_{\epsilon} 2^{-C(j+l)}2^{-jn\big(\frac{1}{p}-\frac{1}{q}\big)},& \text{for all } l>j\epsilon.
    \end{cases}
\end{align*}
Moreover, for $1\leq q \leq 2 \leq p \leq q'$ we have
\begin{align*}
    \begin{cases}
        \|\sum_{l\le j\epsilon}T_{j, l}\|_{L^q\to L^p}\lesssim_{\epsilon}2^{j(m-n\big(\frac{1}{p}-\frac{1}{q}\big)+\zeta\big(\frac{2}{q}-1\big))}\\ 
        \|T_{j, l}\|_{L^q\to L^p}\lesssim_{\epsilon} 2^{-C(j+l)}2^{-jn\big(\frac{1}{p}-\frac{1}{q}\big)},&  \text{for all }l>j\epsilon.
    \end{cases}
\end{align*}
\end{Lem}

\begin{proof}
We begin by proving the geometrically decaying $L^q\to L^p$ estimates
for $T_{j, l}^\nu$ when $1\leq p \leq 2$.  Lemma \ref{kernel lemma} gives us that
\begin{align}\label{pt-est_oio}
|T_{j, l}^\nu f(x)| 
& \lesssim \|(1+2^{jw(k,\rho)}|\nabla_\xi \varphi(x,\xi_j^\nu)-y|)^{N} K_j (x,y)\|_{L^{r'}_y(\Rn)}\\
&\nonumber\qquad\times\Big(\int_{\Rn} \frac{|f(y)|^r}{(1+2^{jw(k,\rho)}|\nabla_\xi \varphi(x,\xi_j^\nu)-y|)^{rN}} \,dy\Big)^{1/r}\\
&\nonumber\lesssim 2^{-lL(C,\eps)+j\varkappa L(C,\eps)}2^{-jw(k,\rho)(n/r'+L(C,\eps))} 2^{-njw(k,\rho)/r} \mathcal{M}_r f(\nabla_\xi \varphi(x,\xi_j^\nu))\\
&\nonumber\lesssim 2^{-lL(C,\eps)+j\varkappa L(C,\eps)}2^{-jw(k,\rho)(n+L(C,\eps))} \mathcal{M}_r f(\nabla_\xi \varphi(x,\xi_j^\nu))\\
&\nonumber\lesssim 2^{-jnk}2^{-C(j+l)} \mathcal{M}_r f(\nabla_\xi \varphi(x,\xi_j^\nu))
\end{align}
for $N>n/r$, $l>j\eps$, and $L(C,\eps)>\max\{C,\frac{\varkappa-w(k,\rho)n+C+C\eps+nk}{w(k,\rho)+\eps}\}.$ By the SND condition and the Hardy-Littlewood maximal theorem we have that
\begin{align*}
    \Vert T_{j, l} f\Vert_{L^p(\Rn)} \lesssim \sum_{\nu=1}^{\mathcal{N}_j} 2^{-jnk} 2^{-C(j+l)}\Vert \mathcal{M}_r f(\nabla_\xi \varphi(\cdot,\xi_j^\nu)) f\Vert_{L^p(\Rn)} 
    \lesssim 2^{-C(j+l)}\Vert f\Vert_{L^p(\Rn)} 
\end{align*}
for $1<r<p\leq \infty$. If $p=1$ then Young's inequality for integral operators and Lemma \ref{kernel lemma} yields
\begin{align*}
    &\Vert T_{j, l}^\nu f\Vert_{L^1(\Rn)} 
    \lesssim 2^{-lL(C,\eps)+j\varkappa L(C,\eps)}2^{-jw(k,\rho)(n+L(C,\eps))}\Vert f\Vert_{L^1(\Rn)}
    \lesssim 2^{-jnk}2^{-C(j+l)}\Vert f\Vert_{L^1(\Rn)}.
\end{align*}
Thus summing this in $\nu$, we have for all $1\leq p\leq 2$ that
\begin{align}\label{Decay_Lp_oio}
    \Vert T_{j, l} f\Vert_{L^p(\Rn)}
    \lesssim 2^{-C(j+l)}\Vert f\Vert_{L^p(\Rn)}.
\end{align}

We now lift the $L^p$ estimate above to geometrically decaying $L^q\to L^p$ bounds of $T_{j, l}^\nu$ in the range $1\leq q\leq p\leq 2$ . To this end, we use Bernstein's inequality for $1\leq q\leq p\leq 2$ and estimate \eqref{Decay_Lp_oio},
\begin{align}\label{geometricBernstein_oio}
    \Vert T_{j, l}^\nu f\Vert_{L^p(\Rn)}& \lesssim \Vert T_{j, l}^\nu \Vert_{L^p\to L^p} \Vert \Psi_j(D) f \Vert_{L^p(\Rn)}\\
    &\nonumber\lesssim 2^{-C(j+l)}2^{-jn\big(\frac{1}{p}-\frac{1}{q}\big)} \Vert f \Vert_{L^q(\Rn)} 
\end{align}    
where $\Psi_j$ is a fattened Littlewood-Paley piece.\\

The next step is to obtain geometrically decaying $L^q\to L^p$ bounds of $T_{j, l}^\nu$ in the range $1\leq q\leq 2\leq p<q'$. Observe that the pointwise estimate \eqref{pt-est_oio} yields that
\begin{align*}
|T_{j, l} f(x)| 
& \lesssim \sum_{\nu=1}^{\mathcal{N}_j}\|(1+2^{jw(k,\rho)}|\nabla_\xi \varphi(x,\xi_j^\nu)-y|)^{N} K_{j,l}^{\nu} (x,y)\|_{L^{\infty}_y(\Rn)}\|f\|_{L^1(\Rn)}\\
& \nonumber\lesssim 2^{-C(j+l)}\|f\|_{L^1(\Rn)}.
\end{align*}
By now interpolating this $L^1\to L^\infty$ estimate and the $L^2\to L^2$ estimate of $T_{j, l}$ we obtain for $1\leq p\leq 2$ that
\begin{align}
    \Vert T_{j, l} f\Vert_{L^{p'}(\Rn)}&\lesssim 2^{-C(j+l)} \Vert f \Vert_{L^{p}(\Rn)}.
\end{align} 
Now interpolating this with \eqref{geometricBernstein_oio} with $p=2$ and using Bernstein's inequality again we obtain
\begin{align}\label{geometricBernstein2_oio}
    \Vert T_{j, l} f\Vert_{L^p(\Rn)} \lesssim 2^{-C(j+l)}2^{-jn\big(\frac{1}{p}-\frac{1}{q}\big)} \Vert f \Vert_{L^q(\Rn)} 
\end{align}    
for $1\leq q\leq 2\leq p\leq q'$.\\

\textit{The case of $\sum_{l\le j\epsilon} T_{j, l}$:\\}

In this case we leverage the geometrically decaying bounds on $T_{j, l}^\nu$ and the $L^2$ boundedness of $T_j$ established in \cite{IMS2}. For a fixed $j$ consider $T_j$, defined as $2^{j(-\zeta-m)}T_j = \sum_l 2^{j(-\zeta-m)}T_{j, l}$, which corresponds to a OIO associated with an amplitude in the class $S_{\rho, \delta}^{\zeta}$ with $L^p$ operator norm of $$\norm{2^{-jm}T_j}_{L^2\to L^2}=2^{j\zeta}.$$ By the $L^2$-boundedness in \cite{IMS2} we establish the inequality $$\|\sum_l T_{j, l}\|_{L^2\to L^2} \lesssim 2^{j(m+\zeta)}.$$
Using the decomposition $$\|\sum_{l\le j\epsilon} T_{j, l}\|_{L^2\to L^2} \leq\| \sum_l T_{j, l}\|_{L^2\to L^2} + \|\sum_{l>j\epsilon} T_{j, l}\|_{L^2\to L^2},$$ 
and the geometrically decaying estimate \eqref{Decay_Lp_oio} we obtain 
\begin{equation}\label{L2oio}
    \|\sum_{l\le j\epsilon} T_{j, l}\|_{L^2\to L^2}\lesssim 2^{j(m+\zeta)}.
\end{equation}
Similarly for $L^1\to L^1$ we obtain
$$\|\sum_{l\le j\epsilon} T_{j, l}\|_{L^1\to L^1} \lesssim 2^{j(m+n(1-\varkappa)/2+\zeta)}.$$

Interpolating these two bounds we obtain for $1\leq p\leq 2$ that
$$\|\sum_{l\le j\epsilon} T_{j, l}\|_{L^p\to L^p} \lesssim 2^{j(m+\frac{n(1-\varkappa)}{2}\big(\frac{2}{p}-1\big)+\zeta)}.$$

 As demonstrated in the previous case, this and Bernstein's inequality for $1\leq q\leq p\leq 2$ yields 
\begin{align*}
    \Vert \sum_{l>j\epsilon} T_{j, l} f\Vert_{L^p(\Rn)}&\lesssim 2^{j(m+\frac{n(1-\varkappa)}{2}\big(\frac{2}{p}-1\big)+\zeta)}2^{-jn\big(\frac{1}{p}-\frac{1}{q}\big)} \Vert f \Vert_{L^q(\Rn)} .
\end{align*} 

We do the $L^\infty$ bound as before, by taking out the supremum of the kernel and summing over $\nu$, this yields the estimate
\begin{align*}
\|T_{j} f\|_{L^{\infty}(\Rn)}
& \lesssim \sum_{\nu=1}^{\mathcal{N}_j}\sup_{x,y\in \Rn}|K_{j,l}^{\nu} (x,y)|\|f\|_{L^1(\Rn)}
 \nonumber\lesssim 2^{j(m+n)}\|f\|_{L^1(\Rn)},
\end{align*}
which combined with the corresponding estimates for $T_{j, l}$ and the rapid decay yields that 
\begin{align*}
\|\sum_{l>j\epsilon} T_{j,l} f\|_{L^{\infty}(\Rn)} \lesssim 2^{j(m+n)}\|f\|_{L^1(\Rn)}.
\end{align*}

By interpolating the $L^1\to L^\infty$ estimate and the $L^2\to L^2$ estimate of $\sum_{l\le j\epsilon} T_{j, l}$, we obtain for $1\leq p\leq 2$ that
\begin{align*}
    \Vert \sum_{l\le j\epsilon} T_{j, l} f\Vert_{L^{p'}(\Rn)}& 2^{j(m+\zeta+n\big(\frac{2}{p}-1\big)+\zeta\big(1-\frac{2}{p}\big))} \Vert f \Vert_{L^{p}(\Rn)}.
\end{align*}
Moreover, \eqref{L2oio} and Bernstein's inequality yields
\begin{align*}
    \Vert \sum_{l\le j\epsilon} T_{j, l} f\Vert_{L^{2}(\Rn)}& 2^{j(m+\zeta-n\big(\frac{1}{2}-\frac{1}{p}\big))} \Vert f \Vert_{L^{p}(\Rn)}.
\end{align*}

Now interpolating these last two estimates we obtain
\begin{align*}
    \Vert \sum_{l\le j\epsilon} T_{j, l} f\Vert_{L^p(\Rn)} \lesssim  2^{j(m-n\big(\frac{1}{p}-\frac{1}{q}\big)+\zeta\big(\frac{2}{q}-1\big))} \Vert f \Vert_{L^q(\Rn)} 
\end{align*}    
for $1\leq q\leq 2\leq p\leq q'$.
\end{proof}

\section{$L^q$-improving estimates for Fourier integral operators}\label{sec:proof_fio}

This section is devoted to showing the $L^q\to L^p$ boundedness of Fourier integral operators. 

\begin{Lem}\label{lplq_fio}
Assume that $a(x,\xi)\in S^{m}_{\rho,\delta}(\Rn)$ for $0< \rho\leq 1$ and $0\leq \delta< 1$, and $\varphi(x,\xi)$ is in the class $\Phi^2$ and is \emph{SND}. Then for $1\leq q\leq p\leq 2$ 
\begin{align*}
    \begin{cases}
        \|\sum_{l\le j\epsilon}T_{j, l}\|_{L^q\to L^p}\lesssim_{\epsilon} 2^{j(m+(\kappa+(n-\kappa)(1-\rho))\big(\frac{1}{p}-\frac{1}{2}\big)+\zeta)}2^{-jn\big(\frac{1}{p}-\frac{1}{q}\big)},\\ 
        \|T_{j, l}\|_{L^q\to L^p}\lesssim_{\epsilon} 2^{-C(j+l)}2^{-jn\big(\frac{1}{p}-\frac{1}{q}\big)},&\text{for all } l>j\epsilon.
    \end{cases}
\end{align*}
Moreover, for $1\leq q \leq 2 \leq p \leq q'$ we have
\begin{align*}
    \begin{cases}
        \|\sum_{l\le j\epsilon}T_{j, l}\|_{L^q\to L^p}\lesssim_{\epsilon} 2^{j(m-n\big(\frac{1}{p}-\frac{1}{q}\big)+\zeta\big(\frac{2}{q}-1\big))}\\ 
        \|T_{j, l}\|_{L^q\to L^p}\lesssim_{\epsilon} 2^{-C(j+l)}2^{-jn\big(\frac{1}{p}-\frac{1}{q}\big)},& \text{for all }l>j\epsilon.
    \end{cases}
\end{align*}
\end{Lem}

\begin{proof}
The proof for the $\sum_{l\le j\epsilon} T_{j, l}^\nu$ case is similar to the one above in Lemma \ref{lplq_oio}, corresponding to the OIOs, with a few modifications. Mainly, one substitutes the use of the sharp $L^2$-boundedness result \cite{IMS2} for OIOs with a corresponding optimal $L^2$-boundedness result in \cite[Theorem 2.7]{DS}. The rest of the changes to the argument involve simple numerical modifications.\\

We turn to proving the geometrically decaying $L^q\to L^p$ estimates
for $T_{j, l}^\nu$.

Define $\mathcal{M}''$ as the Hardy-Littlewood maximal function acting on the function in the $x''$ variable, i.e.
\begin{equation*}
    \mathcal{M}''f(x):= \sup_{\delta>0} \frac{1}{|B(x'',\delta)|} \int_{B(x'',\delta)} |f(y'',x')| \dd y'',
\end{equation*}
where $x=(x'',x')$. We define $\mathcal{M}'$ in a similar manner.\\

Using \eqref{fatsecond} we can write $T^\nu_{j,l}f(x)=T^\nu_{j,l}f^\nu_j(x)$, where $f^\nu_j(x)=\chi_{j}^\nu(D)\Psi_j(D)f(x).$ Now using Lemma 2.22 in \cite{IRS}, take $k=\nu$, $n'=\kappa$ and $ r_1=r_2 = \frac 1{2N} <p$ and note that the spectrum of $f^\nu_j$ is $$\supp \widehat {f_j^\nu} \subset \left \{ (\xi'',\xi')\in \Rl^{\kappa} \times \Rl^{n-\kappa}:\ |\xi''|\leq 2^{j}, |\xi'|\leq 2^{\frac {j}{2}} \right \}.$$ Moreover take $c_{j,\nu}=2^{j}$ and $d_{j,\nu}=2^{\frac {j}{2}}$. Then the conditions of Lemma 2.22 in \cite{IRS} and Lemma \ref{kernel lemma} all hold for $f_j^\nu$, and therefore we have
\begin{align}\label{pt-est_fio}
&|T_{j, l}^\nu f^\nu_j(x)| \\
& \nonumber\lesssim  \|w_j^{\nu}(x,\cdot)^{N}K_j^\nu(x,\cdot) \|_{L^{1}_y(\Rn)}\sup_{y\in\Rl^n} \frac{|f^\nu_j(y)|}{w_j^{\nu}(x,y)^{N}}\\
&\nonumber\lesssim 2^{j(m-1)}) 2^{-lL} [\mathcal{M}''(\mathcal{M}' |f^\nu_j|^{r_1})^{r_2/r_1}]^{1/r_2}(\nabla_\xi \varphi(x,\xi_j^\nu))\\
&\nonumber\lesssim 2^{-j(n-1)/2}2^{-j\left (\frac{n+1}{2}-\frac{n+1}{2p}\right )}2^{-C(j+l)} [\mathcal{M}''(\mathcal{M}' |f^\nu_j|^{r_1})^{r_2/r_1}]^{1/r_2}(\nabla_\xi \varphi(x,\xi_j^\nu))
\end{align}
for $l>j\eps$, $L-N> n$ and $L=L(C,\eps,n)$ sufficiently large.\\

By the SND condition, the Hardy-Littlewood maximal theorem, and \eqref{eq:high} we have
\begin{align}\label{Decay_Lp_fio}
    &\Vert T_{j, l} f\Vert_{L^p(\Rn)} \\
    &\lesssim \nonumber\sum_{\nu=1}^{\mathcal{N}_j} 2^{-j(n-1)/2} 2^{-j\left (\frac{n+1}{2}-\frac{n+1}{2p}\right )} 2^{-C(j+l)}\Vert \mathcal{M}''(\mathcal{M}' |f^\nu_j|^{r_1})^{r_2/r_1} \Vert_{L^{p/r_2}(\Rn)}^{1/r_2} \\
    &\lesssim \nonumber2^{-C(j+l)}2^{-j\left (\frac{n+1}{2}-\frac{n+1}{2p}\right )}\Vert \mathcal{M}' |f^\nu_j|^{r_1}) \Vert_{L^{p/r_1}(\Rn)}^{1/r_1} \\
    &\lesssim \nonumber2^{-C(j+l)}2^{-j\left (\frac{n+1}{2}-\frac{n+1}{2p}\right )}\Vert f^\nu_j \Vert_{L^p(\Rn)} \\
    &\lesssim \nonumber2^{-C(j+l)}\Vert f \Vert_{L^p(\Rn)} 
\end{align}
for $1\leq p\leq \infty$.  We also used that there are roughly $2^{j(n-1)/2}$ terms in the sum over $\nu$.\\

We now lift the $L^p$ estimate above to geometrically decaying $L^q\to L^p$ bounds of $T_{j, l}^\nu$ in the range $1\leq q\leq p<\infty$. To this end, we use Bernstein's inequality for $1\leq q\leq p\leq 2$ and estimate \eqref{Decay_Lp_fio},
\begin{align}\label{geometricBernstein_fio}
    \Vert T_{j, l}^\nu f\Vert_{L^p(\Rn)}& \lesssim \Vert T_{j, l}^\nu \Vert_{L^p\to L^p} \Vert \Psi_j(D) f \Vert_{L^p(\Rn)}\\
    &\nonumber\lesssim 2^{-C(j+l)}2^{-jn\big(\frac{1}{p}-\frac{1}{q}\big)} \Vert f \Vert_{L^q(\Rn)} 
\end{align}    
where $\Psi_j$ is a fattened Littlewood-Paley piece.\\

The next step is to obtain geometrically decaying $L^q\to L^p$ bounds of $T_{j, l}^\nu$ in the range $1\leq q\leq 2\leq p<q'$. Observe that
\begin{align*}
|T_{j, l} f(x)| 
& \lesssim \sum_{\nu=1}^{\mathcal{N}_j}\|(1+2^{jw(k,\rho)}|\nabla_\xi \varphi(x,\xi_j^\nu)-y|)^{N} K_{j,l}^{\nu} (x,y)\|_{L^{\infty}_y(\Rn)}\|f\|_{L^1(\Rn)}\\
& \nonumber\lesssim 2^{-C(j+l)}\|f\|_{L^1(\Rn)}.
 \end{align*}
By interpolating this $L^1\to L^\infty$ estimate and the $L^2\to L^2$ estimate of $T_{j, l}$ we obtain for $1\leq p\leq 2$ that
\begin{align}
    \Vert T_{j, l} f\Vert_{L^{p'}(\Rn)}&\lesssim 2^{-C(j+l)} \Vert f \Vert_{L^{p}(\Rn)}.
\end{align} 
Now interpolating this with \eqref{geometricBernstein_fio} for $p=2$ and using Bernstein's inequality, again we obtain
\begin{align}\label{geometricBernstein_fio2}
    \Vert T_{j, l}^\nu f\Vert_{L^p(\Rn)} \lesssim 2^{-C(j+l)}2^{-jn\big(\frac{1}{p}-\frac{1}{q}\big)} \Vert f \Vert_{L^q(\Rn)} 
\end{align}    
for $1\leq q\leq 2\leq p\leq q'$.
\end{proof}

\section{Proof of the main results}\label{sec:proof}

In this section we give a brief description of the proof of Theorem \ref{main_FIO} and Theorem \ref{main_OIO}. \\

We distil the idea of \cite{BC}, \cite{LMRNC}, and \cite{LSS} to use geometrically decaying $L^q\to L^p$ estimates of pieces of an operator to prove sparse form bounds. We attribute the subsequent lemma to the work of Beltran and Cladek in \cite{BC}, building on earlier work of Lacey, and Spencer \cite{LSS}, and Lacey, Mena, and Reguera \cite{LMRNC}. Using the one-third trick and using shifted dyadic grids, they made calculations that have culminated in the following lemma. Notably, although their original paper did not cast it as a lemma, we have chosen to formalize it as such in our context.

\begin{Lem}[\cite{BC}]\label{sparse domination lift_lemma}
Fix $\eps>0$ and $\delta> \eps$. Then for a sequence of operators $(T_{j,l})_{j,l\geq 0}$ we have
\begin{align*}
\Big<\sum_{j\ge 0}\sum_{l \leq j \epsilon} T_{j,l} f, g\Big>
& \lesssim_{\epsilon}\sum_{j\ge 0}\sum_{\substack{Q\text{ dyadic}: \\ \ell(Q)= 2^{\lfloor{-j\delta+j\epsilon+10\rfloor}}}} \!\!\!\!  |Q|^{1/r+1/s'-1}\Big\|\sum_{l\le j\epsilon}T_{j, l}\Big\|_{L^r\to L^s}|Q| \left<f\right>_{r, Q}\left<g\right>_{s', Q},
\end{align*}
and
\begin{align*}
\Big< \sum_{j \geq 0} \sum_{l > j\epsilon} T_{j,l}f, g \Big> 
& \lesssim_{\epsilon} \sum_{j \geq 0} \sum_{l > j\epsilon} \sum_{\substack{Q \text{ dyadic}: \\ \ell(Q)=2^{\lfloor -j\delta+l+10 \rfloor}}} \!\!\!\! |Q|^{1/r+1/s'-1} \|T_{j,l}\|_{L^r \to L^s} |Q| \left< f \right>_{r,Q} \left< g \right>_{s',Q},
\end{align*}
where the sums are taken over $Q$ in some shifted dyadic grid $\mathcal{D}_v$.
\end{Lem}

\begin{Prop}\label{lifting result}
    Let $1\leq q, p\leq \infty$, $\delta>\eps>0$ and $m<(\delta-\eps)n(1/q-1/p)$ be given. Suppose that for all $l,j\geq 0$ and all $C\gg 1$ (not depending on $l,j$) such that
    \begin{equation}\label{lifing est}
        \begin{cases}
            \|T_{j,l}\|_{L^q \to L^p}\lesssim 2^{-C(j+l)}, \text{ and}&\\
            \|\sum_{l\le j\epsilon}T_{j,l}\|_{L^q \to L^p}\lesssim 2^{jm},
        \end{cases}
    \end{equation}
     holds, then for any compactly supported bounded functions $f, g$ on $\mathbb{R}^n$, there exist sparse collections $\mathcal{S}$ and $\widetilde{\mathcal{S}}$ of dyadic cubes such that
    \begin{equation*}
            \begin{cases}
            |\left<T f, g\right>|\le C(m, q, p)\Lambda_{\mathcal{S}, q, p'}(f, g), \text{ and }\\
            |\left<T f, g\right>|\le C(m, q, p)\Lambda_{\widetilde{\mathcal{S}}, p', q}(f, g).
            \end{cases}
    \end{equation*}
\end{Prop}

\begin{proof}
    Observe that estimate $|\left<T f, g\right>|\le C(m, q, p)\Lambda_{\widetilde{\mathcal{S}}, p', q}(f, g)$ follows by duality from $|\left<T f, g\right>|\le C(m, q, p)\Lambda_{\mathcal{S}, q, p'}(f, g)$. By Lemma \ref{sparse domination lift_lemma}, and the first estimate in \eqref{lifing est}, we have that
    \begin{align*}
        &\Big<\sum_{j\ge 0}\sum_{l > j \epsilon} T_{j,l} f, g\Big>\\
        & \lesssim_{\epsilon}\sum_{j\ge 0}\sum_{l > j \epsilon}\sum_{\substack{Q\text{ dyadic}: \\ \ell(Q)= 2^{\lfloor{-j\delta+j\epsilon+10\rfloor}}}} \!\!\!\!  |Q|^{1/q+1/p'-1}\Big\|T_{j, l}^\nu\Big\|_{L^q\to L^p}|Q| \left<f\right>_{q, Q}\left<g\right>_{p', Q}\\
        & \lesssim_{\epsilon}\sum_{j\ge 0}\sum_{l > j \epsilon} 2^{-C(j+l)}2^{(-j\delta+j\eps)n(1/q-1/p)}
        \sum_{\substack{Q\text{ dyadic}: \\ \ell(Q)= 2^{\lfloor{-j\delta+j\epsilon+10\rfloor}}}} \!\!\!\! 
        |Q| \left<f\right>_{q, Q}\left<g\right>_{p', Q}.
    \end{align*}
    Both sums in $j$ and $l$ converge for sufficiently large $C$. The bound for $\sum_{l \leq j \epsilon} T_{j,l}$ follows from essentially the same arguments, except that we need to take $m<(\delta-\eps)n(1/q-1/p)$. Finally Lemma \ref{universal sparse form} gives us the sparse form bound we desire.
\end{proof}

We now turn to the proof of our main results.

\begin{proof}[Proof of Theorem \ref{main_OIO}]
Note that for $1\leq q\leq p\leq 2$ we have by Lemma \ref{lplq_oio} that
\begin{align*}
    \begin{cases}
        \|\sum_{l\le j\epsilon}T_{j, l}\|_{L^q\to L^p}\lesssim_{\epsilon} 2^{j(m+\frac{n(1-\varkappa)}{2}\big(\frac{2}{p}-1\big)+\zeta)}2^{-jn\big(\frac{1}{p}-\frac{1}{q}\big)},\\ 
        \|T_{j, l}\|_{L^q\to L^p}\lesssim_{\epsilon} 2^{-C(j+l)}2^{-jn\big(\frac{1}{p}-\frac{1}{q}\big)},& l>j\epsilon.
    \end{cases}
\end{align*}
By taking $C$ large enough we have
$$\|T_{j, l}\|_{L^q\to L^p}\lesssim_{\epsilon} 2^{-C(j+l)},\quad l>j\epsilon$$
 for all $1\leq q\leq p\leq 2$. Moreover, to satisfy the hypothesis of Proposition \ref{lifting result} we need that
 \begin{align*}
     m&<n\big(\frac{1}{p}-\frac{1}{q}\big)+\varkappa n\big(\frac{1}{q}-\frac{1}{p}\big)-\frac{n(1-\varkappa)}{2}\big(\frac{2}{p}-1\big)-\zeta\\
     &=-n(1-\varkappa)\big(\frac{1}{q}-\frac{1}{2}\big)-\zeta,
 \end{align*}
since we can take $\eps>0$ arbitrarily small. Under this assumption on $m$, the case when $1\leq q\leq p\leq 2$ in Theorem \ref{main_OIO} follows from Proposition \ref{lifting result}. Moreover, for $1\leq q \leq 2 \leq p \leq q'$ we have
\begin{align*}
    \begin{cases}
        \|\sum_{l\le j\epsilon}T_{j, l}\|_{L^q\to L^p}\lesssim_{\epsilon} 2^{j(m-n\big(\frac{1}{p}-\frac{1}{q}\big)+\zeta\big(\frac{2}{q}-1\big))}\\ 
        \|T_{j, l}\|_{L^q\to L^p}\lesssim_{\epsilon} 2^{-C(j+l)}2^{-jn\big(\frac{1}{p}-\frac{1}{q}\big)},& l>j\epsilon.
    \end{cases}
\end{align*}
The second inequality can be handled as above. For the first one we need that
 \begin{align*}
     m&<n\big(\frac{1}{p}-\frac{1}{q}\big)+\varkappa n\big(\frac{1}{q}-\frac{1}{p}\big)-\zeta\big(\frac{2}{q}-1\big)\\
     &=-n(1-\varkappa)\big(\frac{1}{q}-\frac{1}{p}\big)-\zeta\big(\frac{2}{q}-1\big),
 \end{align*}
since we can take $\eps>0$ arbitrarily small. Under this assumption we then obtain Theorem \ref{main_OIO} in the range $1\leq q \leq 2 \leq p \leq q'$.
\end{proof}

\begin{proof}[Proof of Theorem \ref{main_FIO}]
The proof follows is very similar manner to the one above, the main modification is the following requirements:
 \begin{align*}
     m&<-n(1-\rho)\big(\frac{1}{q}-\frac{1}{2}\big)-\kappa\rho\big(\frac{1}{p}-\frac{1}{2}\big)-\zeta,\quad 1\leq q\leq p\leq 2;\\
     m&<-n(1-\rho)\big(\frac{1}{q}-\frac{1}{p}\big)-\zeta\big(\frac{2}{q}-1\big)\quad 1\leq q \leq 2 \leq p \leq q'.
 \end{align*}
\end{proof}

\bibliographystyle{plain}
\bibliography{references.bib}

\end{document}